\documentclass[aop,MSNbibl,dvips]{arximspdf}
\usepackage{mathbh}
\usepackage{graphicx}

\doi{10.1214/10-AOP558}
\volume{39}
\issue{2}
\pubyear{2011}
\firstpage{549}
\lastpage{586}

\makeatletter

\newcommand{\gap}{\mathbh{g}}
\newcommand{\di}{\mathrm{d}}
\newcommand{\Int}{\operatorname{Int}}
\newcommand{\ev}{\mathbb{E}}
\newcommand{\pr}{\mathbb{P}}
\newcommand{\R}{\mathbb{R}}
\newcommand{\ind}{\mathbf{1}}
\newcommand{\ds}{\displaystyle}
\newcommand{\eps}{\varepsilon}
\newcommand{\nex}{\Theta}
\newcommand{\Z}{\mathcal{Z}}
\newcommand{\T}{T}
\newcommand{\treeroot}{\Lambda}

\newtheorem{theorem}{Theorem}

\newtheorem{corollary}{Corollary}

\newproclaim{definition}{Definition}

\newtheorem{lemma}{Lemma}

\newtheorem{proposition}{Proposition}

\newtheorem{lemA}{Lemma}

\newtheorem{step}{Step}

\makeatother

\begin{document}
\begin{frontmatter}

\title{The algebraic difference of two random Cantor sets: The Larsson family}
\runtitle{The difference of random Cantor sets}

\begin{aug}
\author[A]{\fnms{Michel} \snm{Dekking}\thanksref{t1}\corref{}\ead[label=e1]{F.M.Dekking@tudelft.nl}},
\author[B]{\fnms{K\'{a}roly} \snm{Simon}\thanksref{t1,t2}\ead[label=e2]{simonk@math.bme.hu}}
\and
\author[B]{\fnms{Bal\'{a}zs} \snm{Sz\'{e}kely}\thanksref{t3}\ead[label=e3]{szbalazs@math.bme.hu}}
\thankstext{t1}{Supported in part by the NWO-OTKA common project.}
\thankstext{t2}{Supported by OTKA Foundation \#K71693.}
\thankstext{t3}{Supported by HSN Lab.}
\runauthor{M. Dekking, K. Simon and B. Sz\'{e}kely}
\affiliation{Technical University of Delft, Technical University of
Budapest and \break Technical University of Budapest}
\address[A]{M. Dekking\\
Delft Institute of Applied Mathematics\\
Technical University of Delft\\
Mekelweg 4 \\
2628 CD\\
Delft\\
The Netherlands\\
\printead{e1}}
\address[B]{K. Simon\\
B. Sz\'{e}kely\\
Institute of Mathematics\\
Technical University of Budapest\\
H-1529 P.O. Box 91\\
Budapest\\
Hungary \\
\printead{e2}\\
\phantom{E-mail:\ }\printead*{e3}}
\end{aug}

\received{\smonth{3} \syear{2009}}
\revised{\smonth{3} \syear{2010}}

%
\begin{abstract}
In this paper, we consider a family of random Cantor sets on the
line and consider the question of whether the condition that the sum
of the Hausdorff dimensions is larger than one implies the
existence of
interior points in the difference set of two independent copies.
We give a new and complete proof that this is the case for the random
Cantor sets introduced by Per Larsson.
\end{abstract}


\begin{keyword}[class=AMS]
\kwd[Primary ]{28A80}
\kwd[; secondary ]{60J80}
\kwd{60J85}.
\end{keyword}

\begin{keyword}
\kwd{Random fractals}
\kwd{random iterated function systems}
\kwd{differences of Cantor sets}
\kwd{Palis conjecture}
\kwd{multitype branching processes}.
\end{keyword}

\end{frontmatter}

\section{Introduction}

Algebraic differences of Cantor sets occur naturally in
the context of the dynamical behavior of diffeomorphisms. From
these studies originated a conjecture by Palis and Takens \cite{PT}, relating
the size of the arithmetic difference
\[
C_2-C_1=\{y-x\dvtx x\in C_1,
y\in C_2\}
\]
to the Hausdorff dimensions of the two Cantor sets $C_1$ and $C_2$: if
\begin{equation}\label{P}
\dim_{\mathrm{H}}C_1+\dim_{\mathrm{H}}C_2>1,
\end{equation}
then, \textit{generically}, it should be true that
\[
C_2-C_1 \mbox{ contains an interval}.
\]
For generic dynamically generated \textit{nonlinear} Cantor sets,
this was proven in 2001 by de Moreira and Yoccoz \cite{MY}. The
problem is open for generic linear Cantor sets. The problem was put
into a probabilistic context by Per Larsson in his thesis
\cite{Larssonthesis} (see also \cite{Larsson}). He considers a
two-parameter family of random Cantor sets $C_{a,b}$, and claims to
prove that the Palis conjecture holds for all relevant choices of
the parameters $a$ and $b$. Although the main idea of Larsson's
argument is brilliant, unfortunately, the proof contains significant
gaps and incorrect reasoning. The aim of the present paper is to give a
correct proof of this theorem. The most important error made by
Larsson is as follows: during the construction, a multitype branching process
with uncountably many types appears naturally. The number of
individuals in the $n$th generation
having types which fall into the set $A$ is denoted $\mathcal{Z}_n(A)$
and the probability
measure describing the branching process starting with a single
type-$x$ individual is denoted by $\mathbb{P}_x$.
The argument presented in Larsson's
paper requires that for some positive
$\delta$, $q$, $\rho>1$ and for a set $A$ of which the interior
contains 0,
we have that, \textit{uniformly}, both in $x$ and in $n$, the following
holds:
\begin{equation}\label{16}
{\pr}_x\bigl(\mathcal{Z}_n(A)>\delta\cdot\rho^n\bigr)>q.
\end{equation}

However, the main result in the
theory of general multitype branching processes
\cite{Harris63}, Theorem 14.1, invoked by Larsson implies (\ref{16})
without any uniformity.

Further (as shown in \cite{Chao03}), the idea presented in Larsson's
paper works only in the region (see also Figure \ref{fig:a-b-region}) where
\begin{equation}\label{121}
1-4a-2b+3a^2-6ab>0.
\end{equation}

\begin{figure}

\includegraphics{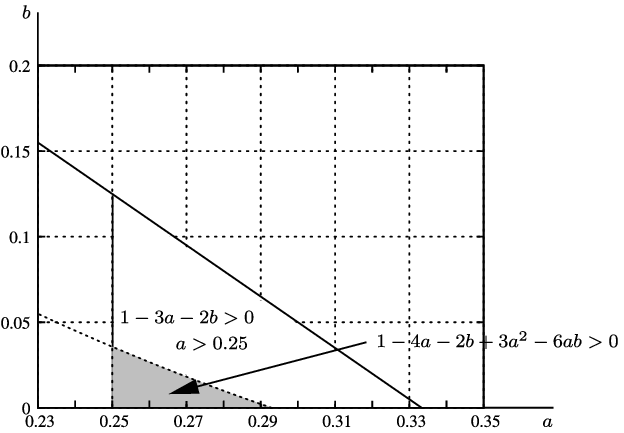}

\caption{Regions described by equations \textup{(\protect\ref{121})} and \textup{(\protect\ref{2})}.}\label{fig:a-b-region}
\end{figure}

Although we use a different setup, the main idea presented here
follows the line of Larsson's proof.

We remark that for linear Cantor sets of a different nature,
the first two authors investigated the same problem in \cite{DekSim}.
Further developments in this direction in \cite{DeDo} lead us to
conjecture that in the
critical case, that is, $\dim_H(C_{a,b})=1/2$, the difference set will
a.s. contain no interval.

\subsection{Larsson's random Cantor sets}\label{LRCS}

It is assumed throughout this paper that
\begin{equation}\label{2}
a>\tfrac{1}{4} \quad \mbox{and} \quad 3a+2b<1.
\end{equation}
The first condition is a growth condition and since
\[
\dim_{\mathrm{H}} C_{a,b}=-\frac{\log2}{\log a},
\]
this condition is equivalent to $\dim_{\mathrm{H}}C_{a,b}>1/2$, which is
equivalent to (\ref{P}).
The second condition is a geometric
condition: Larsson's Cantor set is a natural randomization of the
classical Cantor set; see Figure \ref{fig:Cantor}. In the first
step of the construction, intervals of length $a$ are put into
the intervals $[b,\frac12-\frac{a}{2}]$ and
$[\frac{1}{2}+\frac{a}{2},1-b]$. Dismissing the trivial case
$3a+2b=1$, this obviously requires $3a+2b<1$.
We remark that it is useful to force a forbidden zone of
length at least $a$ in the middle since otherwise the Newhouse
thickness of the Cantor set would be larger than $1$, which yields
an interval in the difference set by Newhouse's theorem (see
\cite{PT}, page~63). The two intervals of length $a$ each have room to
move in an interval of length $\frac12-\frac{a}{2}-b$, that is,
there is a free space of size $\frac12-\frac{a}{2}-b-a$ and we denote
this gap by $\gap$:
\[
\gap:= \frac{1-3a-2b}{2}.
\]

\begin{figure}

\includegraphics{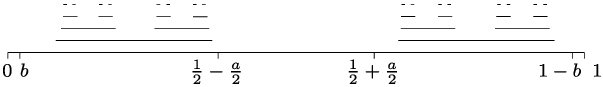}

\caption{The construction of the Cantor set $C_{a,b}$. The figure
 shows $C_{a,b}^1,\dots,C_{a,b}^4$.}\label{fig:Cantor}
\end{figure}


The construction is as follows: first, remove the middle $a$ part,
then the $b$ parts from both the beginning and the end of the unit
interval. Then, place intervals of length $a$ according to a
uniform distribution in the remaining two open spaces
$[b,\frac{1}{2}-\frac{a}{2}]$ and
$[\frac{1}{2}+\frac{a}{2},1-b]$. These two randomly
chosen intervals of length $a$ are called the \textit{level-one
intervals} of the random Cantor set $C_{a,b}$. We write $C_{a,b}^{1}$
for their union. In both of the two level-one intervals, we repeat the
same construction independently of each other and of the previous step.
In this way, we obtain four disjoint intervals of length $a^2$. We
emphasize that, because of independence, the relative positions of
these second level intervals in the first level ones are, in general,
completely different.
Similarly, we construct the $2^n$ level-$n$ intervals of length $a^n$.
We call their union $C_{a,b}^{n}$. Larsson's random Cantor set is then
defined by
\[
C_{a,b}:=\bigcap_{n=1}^\infty C_{a,b}^{n}.
\]
See Figure \ref{fig:Cantor}.

The next theorem was stated by P. Larsson.

\begin{theorem}\label{17} Let
$C_1$, $C_2$ be independent random Cantor sets having the same
distribution as
$C_{a,b}$ defined above. Then, the
algebraic difference $C_2-C_1$ almost surely contains an interval.
\end{theorem}

This paper is organized as follows. In the next section, we give an
elementary proof of the fact that the probability that $C_2-C_1$
contains an interval is either $0$ or $1$.
For the main part of the proof, our starting point is the observation
that $C_2-C_1$ can be viewed as a $45^{\circ}$ projection of the
product set $C_1\times C_2$. This leads, in Section \ref{38}, to the
introduction of the level-$n$ squares formed
as the product of level-$n$ intervals of the Cantor sets $C_1,C_2$.
We remark that Larsson does not use these squares at all. Then, based
on the family of these squares we will construct the intrinsic
branching process and state our \hyperref[15]{Main Lemma}, which will replace
(\ref{16}). In Section \ref{47}, we prove Theorem \ref{17},
assuming the \hyperref[15]{Main Lemma}. In Sections \ref{48}--\ref{sec:main}, we
give a proof of the \hyperref[15]{Main Lemma}.

\section{A 0--1 law}\label{0--1-law}

Undoubtedly, Larsson introduced his Cantor sets as a natural
randomization of the classical
triadic Cantor set. Actually, these sets can also be considered as very
simple examples of
statistically self-similar sets, which permits us to give a simple
proof of the 0--1 law for the
interval property. A set $C$ is \textit{statistically self-similar} if there is a
collection of $m$ random functions
$\{\varphi_1,\dots,\varphi_m\}$ such that
\[
C=\bigcup_{i=1}^m\varphi_i(C_i),
\]
where the $C_i$ are independent random sets with the same distribution
as $C$.
For Larsson's sets, $m=2$ and the random functions are the affine functions
\[
\varphi_1(x)=ax+b+U_1 \quad\mbox{and}\quad \varphi_2(x)=ax+(1+a)/2+U_2,
\]
where $U_1$ and $U_2$ are independent random variables, both uniformly
distributed over $[0,\gap]$.

\begin{proposition}\label{01law}
$\pr(C_2-C_1 \supset I)=0$ or $1$.
\end{proposition}

\begin{pf}
For $1\leq i,j\leq2,$ let $C_{i,j}$ be independent copies of
$C=C_{a,b}$ and let
\[
C_1=\varphi_1(C_{1,1})\cup\varphi_2(C_{1,2}),\qquad C_2=\varphi_1(C_{2,1})\cup\varphi_2(C_{2,2})
\]
be the self-similarity equations for $C_1$ and $C_2$.
We will also write ``$C_2-C_1$ contains an
interval'' equivalently as ``$C_2-C_1$ has nonempty interior.''

Using the facts that for arbitrary subsets $A,B,C$ and $D$ of $\R,$
\[
(A\cup B)-(C\cup D)\supset(A-C)\cup(B-D),
\]
that $\varphi(A-B)=\varphi(A)-\varphi(B)$ for affine functions
$\varphi\dvtx\R\rightarrow\R$
and that affine functions are continuous,
we can set up the following chain of (in)equalities:
\begin{eqnarray*}
p&:=&\pr(C_2-C_1 \supset I)\\
&=& 1- \pr\bigl(\Int(C_2-C_1) = \varnothing\bigr)\\
&\ge& 1- \pr\bigl(\Int\bigl(\varphi_1(C_{2,1})-\varphi_1(C_{1,1})\bigr)=\varnothing,\Int\bigl(\varphi_2(C_{2,2})-\varphi_2(C_{1,2})\bigr)= \varnothing\bigr)\\
&=& 1- \pr\bigl(\Int\bigl(\varphi_1(C_{2,1})-\varphi_1(C_{1,1})\bigr)=\varnothing\bigr)\pr\bigr(\Int\bigr(\varphi_2(C_{2,2})-\varphi_2(C_{1,2})\bigr)= \varnothing\bigr)\\
&=& 1- \pr\bigl(\Int\bigl(\varphi_1(C_{2,1}-C_{1,1})\bigr) =\varnothing\bigr)\pr\bigl(\Int\bigl(\varphi_2(C_{2,2}-C_{1,2})\bigr)= \varnothing\bigr)\\
&=& 1- \pr\bigl(\Int\bigl(C_{2,1}-C_{1,1}) =\varnothing\bigr)\pr\bigl(\Int(C_{2,2}-C_{1,2})= \varnothing\bigr)\\
&=& 1-(1-p)^2.
\end{eqnarray*}
This implies that $p\le p^2$ and hence $p=0$ or 1.
\end{pf}

\section{\texorpdfstring{Notation and the \protect\hyperref[15]{Main Lemma}}{Notation and the Main Lemma}}\label{46}

In the remainder of the paper, we fix a pair $(a,b)$ satisfying
condition (\ref{2}) and
always deal with Larsson's Cantor sets, so we will suppress the labels $a,b$.

\subsection{The geometry of the algebraic difference $C_2-C_1$}\label{38}

The $45^{\circ}$ projection of a point $(x_1,x_2)\in\mathbb{R}^2$ onto
the $x_2$-axis is denoted by $\operatorname{Proj}_{45^{\circ}}$. That is,
\[
\operatorname{Proj}_{45^{\circ}}(x_1,x_2):=x_2-x_1.
\]

The following trivial fact is the motivation for constructing our
branching process of labeled squares:

\[
x \in \operatorname{Proj}_{45^{\circ}}(C_1\times C_2)\quad \mbox{if and only if}\quad x \in C_2-C_1.
\]

So,
\[
C_2- C_1=\bigcap_{n=0}^{\infty}\operatorname{Proj}_{45^{\circ}}(C_{1}^{n}\times C_{2}^{n}).
\]

We can naturally label the squares in $C_{1}^{n}\times C_{2}^{n}$ as follows:
we call the upper-left first level square $Q_1$ and continue labeling
the first level squares $Q_2,Q_3 ,Q_4$ in the clockwise direction;
then, within each of these squares, we continue in this way; see Figure
\ref{fig:only_squares}.

\begin{figure}

\includegraphics{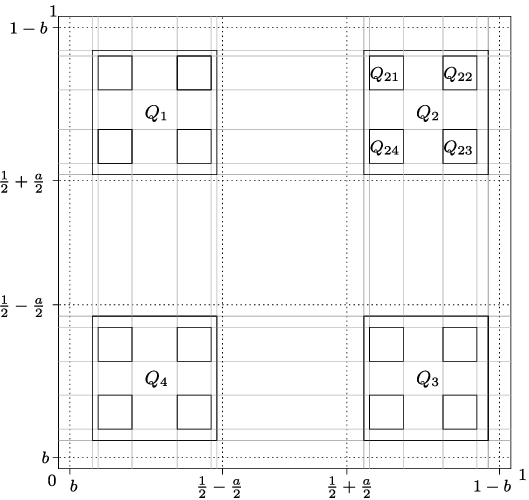}

\caption{The first level squares $Q_1,\dots,Q_4$ and four second
level squares $Q_{21}$, $Q_{22}$, $Q_{23}$, $Q_{24}$.}\label{fig:only_squares}
\end{figure}

For an $x \in[-1,1],$ we write $e(x)$ for that line with
slope $1$ which intersects the vertical axis at $x$. As we observed
above
\begin{equation}\label{3}
x \in C_2-C_1 \quad\mbox{if and only if}\quad e(x )\cap(C_1\times
C_2)\ne\varnothing.
\end{equation}

Fix $x $ and an arbitrary $n$.
Let $\mathcal{S}_n$ be the set of all $a^n\times a^n$ squares
contained in $[0,1]^2$. Note that for every $Q\in\mathcal{S}_n$, by the statistical
self-similarity of the construction, the probability of the event $e(x)\cap(Q\cap(C_1\times C_2))\ne\varnothing$ conditional on
$Q\subset C_1^n\times C_2^n$ is equal to the probability of
the event $e(\Phi)\cap(C_1\times C_2)\ne\varnothing$, where we
construct $\Phi=\Phi(Q,x)$ as follows: we rescale the square $Q$
(which is an $a^n\times a^n$ square) by the factor $1/a^n$, then
we choose $\Phi$ such that the line segment $e(\Phi)\cap[0,1]^2$ is
the rescaled copy of $e(x )\cap Q$; see Figure \ref{fig:rescale}. More
precisely, if $(u,v)$ is the lower-left corner of $Q$, that is,
$Q=[u,u+a^n]\times[v,v+a^n]$, then we define
\begin{equation}\label{eq:Phi}
\Phi(Q,x ):=\cases{
\displaystyle{\frac{u-v+x }{a^n}}, &\quad if $e(x$) intersects $Q$,\cr
\displaystyle{\Theta}, &\quad otherwise,
}
\end{equation}
where $\Theta$ is a symbol representing the emptiness of the intersection.
Observe that $\Phi(Q,x )>0 $ if and only if the
center of $Q$ is located below the line $e(x )$ and $e(x)$ meets $Q$.
Further, $\Phi
(Q,x )=1$ if $e(x )$ intersects $Q$ at the upper-left
corner and $\Phi(Q,x )=-1$ if $e(x )$ intersects $Q$ at the
lower-right corner.

\begin{figure}\label{golden}

\includegraphics{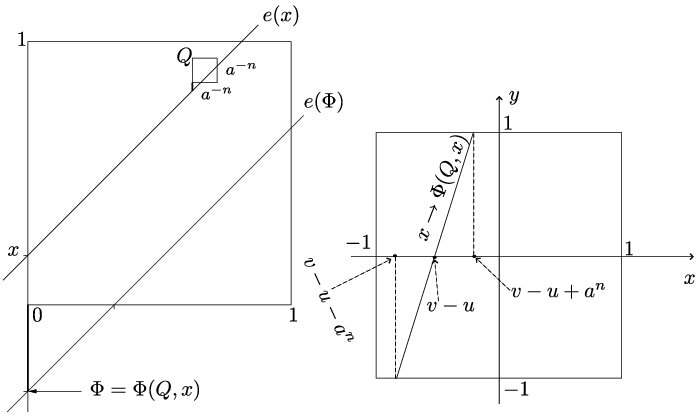}

\caption{A level-$n$ square $Q$ and its rescaled type $\Phi(Q,x)$.}\label{fig:rescale}
\end{figure}

\subsection{The probability space}\label{probab_space_of_Q}

We write $\mathcal{T}:=\bigcup_{n=0}^{\infty}\{1,2\}^n$
for the dyadic tree,
with nodes $\underline{i}_n=i_1i_2\dots i_n,$ where $i_k$ is $1$ or
$2$, and root
$\treeroot$.
For the construction of Larsson's Cantor set, the probability space is
$\Omega_1=[0,\gap]^\mathcal{T}$ [recall that $\gap
=(1-3a-2b)/2$]. An element of $\Omega_1$ is denoted by $U$, that is,
the value at the node $i_1i_2\dots i_n$ is $U_{i_1i_2\dots i_n}$. The
corresponding
$\sigma$-algebra is
$\mathcal{B}_1:=\prod_{\mathcal{T}}\mathcal{B}[0,\gap]$.
Finally, the
probability measure for Larsson's Cantor set is
\[
{\pr}_1:=\delta_0\times\prod_{\mathcal{T}\setminus\{\treeroot\}}\operatorname{Uniform}[0,\gap],
\]
where $\delta_0$ is the Dirac mass at 0 associated with the mass at
the root $\Lambda$. Note that the randomness starts at level 1.
So, the probability space
for $C_1\times C_2$ is as follows:
\begin{equation}\label{18}
\Omega:=\Omega_1\times\Omega_1,\qquad\mathcal{B}:=\mathcal{B}_1\times\mathcal{B}_1,\qquad{\pr}:={\pr}_1\times{\pr}_1.
\end{equation}

An element of $\Omega$ is a pair of labeled binary trees. The $4^n$
level-$n$ pairs of indices $(i_1i_2\dots i_n, j_1j_2\dots j_n)$ are naturally
associated with level-$n$ squares $Q'_{(i_1i_2\dots i_n, j_1j_2\dots
j_n)}$ of size $a^n\times a^n$ whose relative \vspace{1pt} positions are  given by
$U_{i_1i_2\dots i_n}$ and $U_{j_1j_2\dots j_n}$.
Note, however, that (to simplify the notation) we have given new
indices to these
squares and positions: $Q_1:=Q'_{1,2}$, $Q_2:=Q'_{2,2}$,
$Q_3:=Q'_{2,1}$, $Q_4:=Q'_{1,1}$ and similarly for higher order squares
and their positions (see Figure \ref{fig:only_squares}).

\subsection{The branching process}\label{116}

On the probability space $\Omega,$ we define a multitype branching
process $\mathcal{Z}=(\mathcal{Z}_n)_{n=0}^\infty$.
For a Borel set $A,$ the natural number $\mathcal{Z}_n(A)$ represents
the number
of objects in generation $n$ whose type falls into the set $A$. The
type space $T$ is a subset of $[-1,1]$, but for the moment we can think
of $T=[-1,1]$. The objects of the $n$th generation are squares $Q\in
\mathcal{S}_n$ and, given a fixed $x\in[-1,1]$,
their type is $\Phi(Q,x)$, as defined in (\ref{eq:Phi}). Note that
although we speak of $\Theta$ as a type, it is \textit{not} an element
of $T$.

The process $(\mathcal{Z}_n)$ is a Markov chain whose states are
collections of squares labeled by their types.
The transition mechanism is as described in Section \ref{38}. The
initial condition of the chain is the square
$[0,1]\times[0,1]$, with type $x$ (also called the \textit{ancestor} of
the branching process). As usual, we then write, for $n\ge1,$
\begin{eqnarray*}
&&{\pr}_x\bigl(\mathcal{Z}_{n} (A_1)=r_1,\dots, \mathcal{Z}_{n}(A_k)=r_k \bigr)\\
&&\qquad={\pr} \bigl(\mathcal{Z}_{n} (A_1)=r_1,\dots, \mathcal{Z}_{n}(A_k)=r_k| \mathcal{Z}_{0}(\{x\})=1\bigr)
\end{eqnarray*}
for all $k\ge1$, $A_1,\dots,A_k \subset T$ and nonnegative integers
$r_1,\dots,r_k$.

A collection of squares all with type $\Theta$ is an absorbing state:
it only generates squares with type $\Theta$.
This is obvious from the definition of $\Phi(Q,x)$, but we will extend
this property to the case of smaller type spaces $T$,
where, by definition, a square has type $\Theta$ if its type is not in
$T$ (this will be further explained in Section~\ref{sec:6.1}).

A major role in our analysis is played by the expectations $\ev
_x[\mathcal{Z}_n(A)]$ for $A\subset T$, $n\ge1$.
Let us define, for $i=1,2,3,4,$
\begin{equation}\label{eq:Z1}
\ds\mathcal{Z}_1^i(A) = \cases{
1, & \quad if $\Phi(Q_i,x)\in A$,\cr
0, & \quad otherwise.
}
\end{equation}
Then, $\mathcal{Z}_1(A)=\mathcal{Z}_1^1(A)+\cdots+\mathcal
{Z}_1^4(A)$ and so
\begin{eqnarray*}
\ev_x[\mathcal{Z}_1(A)]&=&\int_\Omega\mathcal{Z}_1(A)\,\di\pr_x=\int_\Omega\sum_{i=1}^4\mathcal{Z}_1^i(A)\,\di\pr_x\\
&=&\sum_{i=1}^4\pr_x\bigl(\Phi(Q_i,x)\in A\bigr)=\sum_{i=1}^4\int_A f_{x,i}(y)\,\di y,
\end{eqnarray*}
where the $f_{x,i}$ are the densities of the random variables $\Phi
(Q_i,x)$ (apart from an atom in
$\Theta$). In Section \ref{sec:dens}, these densities will be
determined explicitly. It follows that for $n=1,$
\[
M_n(x,A):=\ev_x[\mathcal{Z}_n(A)]\label{115}
\]
has a density $m_1(x,y)$, called the \textit{kernel} of the branching
process, given by
\begin{equation}\label{m-kernel}
m(x,y):=m_1(x,y)=\sum_{i=1}^4 f_{x,i}(y).
\end{equation}
We remark that if $M_1$ has a density, then $M_n$
also has a density. Let us write $m_n(x ,\cdot)$
for the density of $M_n(x ,\cdot)$. The branching structure of
$\mathcal{Z}$ yields (see \cite{Harris63}, page 67)
\begin{equation}\label{eq:rec-mn}
m_{n+1}(x,y)=\int_{T}m_n(x,z)m_1(z,y)\,\di z.
\end{equation}
The main problem to be solved is that the natural choice of $T=[-1,1]$
as type space does not work because of condition (C) below and
because we need the uniformity alluded to in equation (\ref{16}).

Since the definition of $T$ is complicated, we postpone it to Section
\ref{sec:up_kernel_simple}. However, here we collect the most
important properties of $T$:
\begin{longlist}
\item[(A)] $T$ is the disjoint union of finitely many closed intervals;
\item[(B)] there exists a $K>0$ such that $[-K,K]\subset T$;
\item[(C)] the kernel $m_n(x,y)$ defined in (\ref{eq:rec-mn}) is
uniformly positive on $T\times T$ [see condition (\ref{cond:dens})
below] and it has Perron--Frobenius eigenvalue greater than $1$ [see
condition (\ref{cond:rho}) below].
\end{longlist}

\subsection{The asymptotic behavior of the branching process $\mathcal{Z}$}\label{117}

We will prove in Sections \ref{sec:up_kernel_simple}, \ref{112} and
\ref{sec:up_kernel} that there exists an integer
$n_0$ such that $m_{n_0}$ is a uniformly bounded function, that is,
there exist $0<a_{\min}<a_{\max} $ such that for all $x ,y\in T,$ we have
{\def\theequation{C1}
\begin{equation}\label{cond:dens}
0<a_{\min}\leq m_{n_0}(x ,y)\leq a_{\max}<\infty.
\end{equation}
}\vspace*{-\baselineskip}

\noindent In the next step, we consider the following two operators:
\setcounter{equation}{10}
\begin{equation}\label{10}
g(x)\mapsto\int_{\mathbb{R}}m_1(x ,y)\cdot g(y)\,\di y,\qquad h(y)\mapsto\int_{\mathbb{R}}h(x)\cdot m_1(x ,y)\,\di x.
\end{equation}

We cite the following theorem from \cite{Harris63}, Theorem 10.1.
\begin{theorem}[(Harris)]\label{Harris:10.1}
It follows from (\ref{cond:dens}) that the operators in (\ref{10}) have a
common dominant eigenvalue $\rho$. Let $\mu(x)$ and $\nu(y)$ be the
corresponding eigenfunctions of the first and second operator in
(\ref{10}), respectively. Then, the functions $\mu(x)$ and $\nu(y)$
are bounded and uniformly positive. Moreover, apart from a scaling,
$\mu$ and $\nu$ are
the only nonnegative eigenfunctions of these operators. Further, if we
normalize
$\mu$ and $\nu$ so that $\int\mu(x)\nu(x)\,\di x=1$, which will be
henceforth assumed, then, for all $x,y \in T,$ as $n\to\infty,$
\[
\bigg|\frac{m_n(x,y)}{\rho^n}-\mu(x)\nu(y)\bigg|\le C_1\mu(x)\nu
(y)\Delta^n,
\]
where the bound $\Delta<1$ can be taken independently of $x$
and $y$,
and the constant $C_1$ is independent of $x,y$ and $n$.
\end{theorem}

Later in this paper, we will prove that in our case, this Perron--Frobenius
eigenvalue is greater than one:
{\def\theequation{C2}
\begin{equation}\label{cond:rho}
\rho>1.
\end{equation}
}\vspace*{-\baselineskip}

\noindent Using Theorem \ref{Harris:10.1}, Harris proves that $\mathcal{Z}_n
(A)$ in fact
grows exponentially with rate~$\rho$. Introducing
\[
W_n(A):=\frac{\mathcal{Z}_n (A)}{\rho^n},
\]
he obtains
(see \cite{Harris63}, Theorem 14.1)
the following result.
\begin{theorem}[(Harris)]\label{thmA:14.1}
If
{\def\theequation{C3}
\begin{equation}\label{cond:secmom}
\sup_{x\in T }\ev_x [\mathcal{Z}_1 (T)^2]<\infty,
\end{equation}
}\vspace*{-\baselineskip}

\noindent then it follows from (\ref{cond:dens}) and (\ref{cond:rho})
that for all $x \in T,$
\setcounter{equation}{11}
\begin{equation}\label{13}
{\pr}_x \Bigl(\lim\limits_{n\to\infty}W_n(A)=:W(A)\Bigr)=1.
\end{equation}
Further, for every Borel measurable $A\subset T$ with
$\mathcal{L}{\mathbf{eb}}_1(A)>0,$ we have
\begin{equation}\label{eq:star}
{\pr}_x \bigl(W(A)>0\bigr)>0.
\end{equation}
Moreover, let $A$ and $B$ be subsets of ${T}$ such that their
Lebesgue measures are positive. Then, the relation
\[
W(B)=\frac{\int_B \nu(y)\,\di y}{\int_A \nu(y)\,\di y}W(A)
\]
holds $\pr_x$ almost surely for any $x\in T$.
\end{theorem}

We are going to use this theorem to prove our \hyperref[15]{Main Lemma}, which
summarizes everything we need concerning our branching process. Roughly
speaking, the \hyperref[15]{Main Lemma} says that for the branching process
associated to Larsson's Cantor set, the statement in
Theorem~\ref{thmA:14.1} holds uniformly both in $n$ and $x$ for an
appropriately chosen small interval of $x$'s.

\begin{main*}\label{15}
There exist positive numbers $\delta$ and $q$, an $N\in\mathbb{N}$
and a small interval $[-K,K]\subset
T$ centered at the origin such that the following
inequality holds:
\begin{equation}\label{eq:mainZ}
\inf_{n\geq N}\inf_{x \in[-K,K]}\pr_x\bigl(\mathcal{Z}_n([-K,0])>\delta\rho^n,\mathcal{Z}_n([0,K])>\delta\rho^n\bigr) \geq q.
\end{equation}
\end{main*}

\section{\texorpdfstring{The proof of Theorem \protect\ref{17}}{The proof of Theorem 1}}
\label{47}
In Section \ref{38}, we defined the type of a square $Q$ by means of
its intersection with a line $e(x )$.
Here, we will elaborate on this intersection.

\subsection{Nice intersection of a square with a line $e(x)$}

We say that a square $Q$ has a \textit{nice intersection} with
$e(x)$ if
\[
\Phi(Q,x )\in[-K,K],
\]
where $K$ comes from \hyperref[15]{Main Lemma}. For small $K,$ this means that the center of
$Q$ is close to the line $e(x)$.

Let $\mathcal{A}^{0}=\{[0,1]^2\}$, $\mathcal{A}^{n}$ be the set $\{Q\in\mathcal{S}_{n}\dvtx Q\subset C_1^{n}\times C_2^{n}\}$ and $\mathcal
{A}_{x }^{n}$ be the
set of squares from $\mathcal{A}^{n}$
having nice intersection with $e(x)$. That is, for $x \in\T$ and
$n\geq1,$ we define
\[
\mathcal{A}_{x }^{n}:=\{Q\in\mathcal{A}^n\dvtx |\Phi(Q,x)|\leq K\}.
\]
Moreover, for $m\geq0$ and a square $Q\in\mathcal{A}^m_x$, we write
$l^+_n (Q,x)$ and ($l^-_n (Q,x)$)
for the numbers of level-$(m+n)$ squares contained in $Q$ which have
nice intersection with $e(x)$ with center below and above the line
$e(x),$ respectively. That is, for a $Q=Q_{i_1\dots i_m}$, let
\[
l^+_n (Q,x)=\#\{Q_{i_1\dots i_mj_{1}\dots j_{n}}\in\mathcal
{A}^{m+n}\dvtx 0\leq\Phi(Q_{i_1\dots i_mj_1\dots j_{n}},x)\leq K\}.
\]
Similarly, let
\[
l^-_n (Q,x)=\#\{Q_{i_1\dots i_mj_{1}\dots j_{n}} \in\mathcal
{A}^{m+n}\dvtx -K\leq\Phi(Q_{i_1\dots i_mj_1\dots j_{n}},x)\leq 0\}.
\]
Finally, for every $n\geq1$, $x\in\T$ and $Q\in\mathcal
{A}^m_x$, we define the event
\[
A_n(Q,x ):=\{ l^-_n (Q,x)>\delta\rho^n, l^+_n (Q,x)>\delta\rho
^n\},
\]
where $\delta$ comes from the \hyperref[15]{Main Lemma}. Note that the self-similarity of the construction of the squares and
the \hyperref[15]{Main Lemma} for the underlying branching process imply the
following: for $n\geq N$ and a square $Q\in\mathcal{S}_m$, we have
\begin{eqnarray}\label{a}
&&\pr(A_n(Q,x)| Q\in\mathcal{A}^m_x )\nonumber\\[-8pt]\\[-8pt]
&&\qquad=\pr_{\Phi(Q,x)}\bigl(\mathcal{Z}_n([-K,0])>\delta\rho^n,\mathcal{Z}_n ([0,K])>\delta\rho^n\bigr)\geq
q.\nonumber
\end{eqnarray}

\subsection{The difference set $C_2-C_1$ contains an interval with
positive ${\pr}$ probability}

We introduce the interval
\[
I:=[- K a^N, K a^N]
\]
with $N$ and $K$ from the \hyperref[15]{Main Lemma}.
Note that $|I|:=\mathcal{L}\mathbf{eb}_1(I)=2Ka^N$.

Our goal is to prove that
\[
{\pr} (C_2-C_1\supset I) > 0.\label{goal}
\]

First, we divide the interval $I$ into $4^{2N}$ intervals $I_{i_1}$
of equal length with indices $\pm1,\ldots,\pm\frac{1}{2}4^{2N}$.
Then, we divide all of these intervals into
$4^{3N}$ intervals $I_{i_1i_2}$ of equal length. If we have
already defined the $(k-1)$th level intervals, then we define the
$k$th level intervals $I_{i_1\ldots i_k}$ by subdividing each
$(k-1)$th level interval $I_{i_1\ldots i_{k-1}}$ into $4^{(k+1)N}$
intervals of equal length with indices $\pm1,\ldots,\pm\frac
{1}{2}4^{(k+1)N}$. We denote the center of $I_{i_1\ldots i_k}$ by
$z_{i_1\ldots i_k}$.
That is,
\[
I_{i_1\dots i_k}=\bigl[z_{i_1\dots i_k}-K a^N4^{-[2+\cdots+(k+1)]N},z_{i_1\dots i_k}+K a^N4^{-[2+\cdots+(k+1)]N}\bigr],
\]
where the $z_{i_1\dots i_k}$ are equally spaced in $I_{i_1\ldots
i_{k-1}}$.

Note that the interval $I_{i_1\ldots i_k}$ has
length
\begin{equation}\label{200}
|I_{i_1\ldots i_k}|=2K a^N4^{-[2+\cdots+(k+1)]N}<2Ka^{g_k},
\end{equation}
where we put
\[
g_k:=(1+\cdots+(k+1))N=\tfrac12(k+1)(k+2)N.
\]
In the following, we will go from generation $g_{k-1}$ to generation
$g_{k} $.
\begin{definition}\label{45}
We say that the event ${B_k(z_{i_1\ldots i_k})}$ occurs if there exists
some square
$Q\in\mathcal{A}^{g_{k-1}}$, itself having nice intersection with
$e(z_{i_1\ldots i_k})$,
such that $A_{(k+1)N}(Q,z_{i_1\ldots i_k})$ holds---cf. Figure \ref{Delft13}.
In formulae,
\begin{equation}\label{B}
B_k(z_{i_1\ldots i_k})= \bigcup\limits_{Q\in\mathcal
{A}_{z_{i_1\ldots i_k}}^{g_{k-1}}}
A_{(k+1)N}(Q, z_{i_1\ldots i_k}).
\end{equation}
\end{definition}

The following lemma is one of the key statements of the argument.

\begin{lemma}\label{lem:B_k}
Assume that $B_k(z_{i_1\ldots i_k})$ occurs with the square $Q$. Let
$\mathcal{Q}^+$ and $\mathcal{Q}^-$
be the collections of level-$g_k$ squares within $Q$
having nice intersection with $e(z_{i_1\ldots i_k})$ with center below
and above the line $e(z_{i_1\ldots i_k}),$ respectively. Then,
\begin{longlist}[(1)]
\item[(1)]
\[
\operatorname{Proj}_{45^{\circ}}\biggl(\bigcup_{\widetilde{Q}\in\mathcal{Q}^+}
\widetilde{Q}\biggr)\supset I_{i_1\ldots i_k},\qquad \operatorname{Proj}_{45^{\circ}}\biggl(\bigcup_{\widetilde{Q}\in\mathcal{Q}^-}
\widetilde{Q}\biggr)\supset I_{i_1\ldots i_k}.
\]
\item[(2)] For every
$i_{k+1}=\pm1,\ldots,\pm\frac{1}{2}4^{(k+2)N}$, the line
$e(z_{i_1\ldots
i_k i_{k+1}})$ has nice intersection with all squares from either
$\mathcal{Q}^+$ or $\mathcal{Q}^-$. Thus, the line
$e(z_{i_1\ldots
i_k i_{k+1}})$
has nice intersection with at least
$\delta\rho
^{(k+1)N}$ squares contained in $Q$ such that either all have center
below the line $e(z_{i_1\ldots i_k})$ or all have center above the line
$e(z_{i_1\ldots i_k})$.
\end{longlist}
\end{lemma}
\begin{figure}

\includegraphics{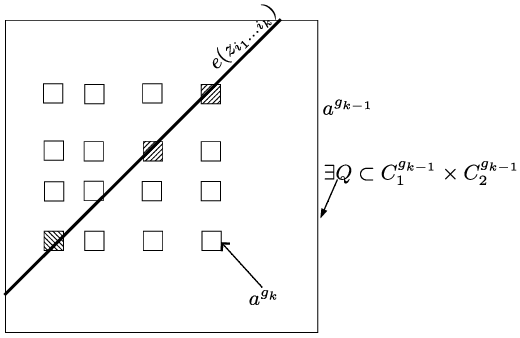}

\caption{Event $B_k(z_{i_1\dots i_k})$: there is a level-$g_{k-1}$
square $Q$ in which the number of striped level-$g_k $ squares
(the nicely intersecting ones) is at least $\delta\rho
^{N(k+1)}$, both for the squares with center above and the squares with center
below the line $e(z_{i_1\dots i_k})$.}\label{Delft13}
\end{figure}
\begin{pf} Choose an arbitrary $y\in I_{i_1\dots i_k}$. Without loss
of generality, we may assume that $y\leq z_{i_1\dots i_k}$.
Then, to show both (1) and (2), it is enough to prove that $e(y)$ has
nice intersection with
all squares from $\mathcal{Q}^+$.

Fix an arbitrary $Q\in\mathcal{Q}^+$. By the definition of $\mathcal
{Q}^+$, the square $Q$ is a level-$g_k$ square such that its lower-left
corner is in between the parallel lines $e(z_{i_1\dots i_k})$
and $e(z_{i_1\dots i_k}-Ka^{g_k})$.
So, for every point $y^*\in[z_{i_1\dots i_k}-Ka^{g_k},z_{i_1\dots
i_k}],$ the line $e(y^*)$ has nice intersection with $Q$; see Figure
\ref{fig:square_covers}.

\begin{figure}

\includegraphics{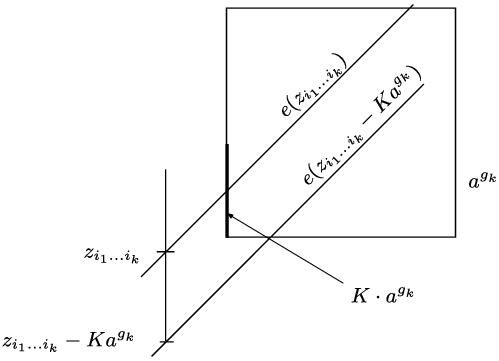}

\caption{Nice intersections.}\label{fig:square_covers}\vspace*{2pt}
\end{figure}

To show that for any $y\in I_{i_1\dots i_k}\cap(-\infty,z_{i_1\dots
i_k}],$ $e(y)$ has nice intersection with
all squares from $\mathcal{Q}^+,$ it is enough to prove that
\[
I_{i_1\dots i_k}\cap(-\infty,z_{i_1\dots i_k}]\subset[z_{i_1\dots
i_k}-Ka^{g_k},z_{i_1\dots i_k}],
\]
based on the previous paragraph. However, since
\[
|I_{i_1\dots i_k}\cap(-\infty,z_{i_1\dots i_k}]|=\tfrac{1}{2}|I_{i_1\ldots i_k}|< K a^{g_k},
\]
this follows using (\ref{200}).
\end{pf}

\begin{definition}
Let $ E_0:=A_N([0,1]^2,0)$ and let $ E_k:=\bigcap_{i_1\ldots
i_k}B_k(z_{i_1\ldots i_k}).$
\end{definition}

\begin{lemma}\label{lem:subset}
The following inequality holds:
\begin{equation}\label{almost}
\pr(C_2-C_1\supset I)\geq q\prod\limits_{k\geq 1}\pr(E_k|E_{k-1}).
\end{equation}
\end{lemma}

\begin{pf}
Using the fact that $ I=[-Ka^N,Ka^N]=\bigcup_{i_1\ldots i_k} I_{i_1\ldots i_k}$,
it follows immediately from Lemma \ref{lem:B_k} that if the event
$E_k$ holds, then the event
\[
S_k:=\{\operatorname{Proj}_{45^{\circ}}(C_1^{g_k}\times C_2^{g_k}\label{ek})\supset I\}
\]
will hold.
Therefore, $E_k\subset S_k$. Since the sets $C_1^{g_k}\times C_2^{g_k}$
are decreasing, we obtain that $S_k\supset S_{k+1}$. Thus,
\begin{eqnarray*}\label{eq:CCI}
\pr(C_2-C_1\supset I)&=&\displaystyle\pr\biggl(\bigcap_{k\geq1} S_k\biggr)=\lim\limits_{k\to\infty}\pr(S_k)\geq\inf_{k\geq1} \pr(E_k)\\
&\geq&\displaystyle\pr(E_0)\prod\limits_{k\geq1}\pr(E_k|E_{k-1}).
\end{eqnarray*}
The last inequality holds since
\begin{eqnarray*}
\pr(E_0)\prod\limits_{i\geq1}\pr(E_i|E_{i-1})&\leq& \pr(E_0)\pr(E_1|E_0)\cdots\pr(E_k|E_{k-1})\\
&=&p\pr(E_kE_{k-1})\leq\pr(E_k),
\end{eqnarray*}
where
\[
p=\frac{\pr(E_0)}{\pr(E_0)}%
\frac{\pr(E_1E_0)}{\pr(E_1)}\cdots%
\frac{\pr(E_{k-1}E_{k-2})}{\pr(E_{k-1})}\leq1.
\]
Since the \hyperref[15]{Main Lemma} yields $ \pr(E_0)\geq q,$ one obtains the
statement of the lemma.
\end{pf}

In Lemma \ref{lem:E_k}, we give a lower bound for
$\pr(E_k|E_{k-1})$ for every $k$.

\begin{lemma}\label{lem:E_k}
For any $k\geq1,$ we have
\[
\pr (E_k|E_{k-1})\geq
1-4^{2N+\cdots+(k+1)N}(1-q)^{\delta\rho^{kN}}.
\]
\end{lemma}

\begin{pf}
We recall that $E_k$ was defined as
\[
E_k:=\bigcap\limits_{i_1\ldots i_k}B_k(z_{i_1\ldots i_k}).
\]
Therefore, we have to prove that
\[
\mathbb{P}\biggl(\bigcup_{i_1\dots i_k}B_{k}^{c}(z_{i_1\ldots i_k})\Big|E_{k-1}\biggr)\leq4^{2N+\cdots+(k+1)N}(1-q)^{\delta\rho^{kN}}.\label{40}
\]
Note that the number of indices $i_1\dots i_k$ on the left-hand side is
equal to $4^{2N+\cdots+(k+1)N}$. Therefore, it is enough to show that
for each index $i_1\dots i_k,$ we have
\[
\mathbb{P}(B_{k}^{c}(z_{i_1\dots i_k})|E_{k-1})\leq(1-q)^{\delta\rho^{kN}}.\label{41}
\]
By Definition \ref{45}, to see this, we have to prove that
\begin{equation}\label{42}
\mathbb{P}\biggl(\bigcap\limits_{Q\in{\mathcal{A}}_{z_{i_1\ldots i_k}}^{g_{k-1}}}A^{c}_{(k+1)N}(Q, z_{i_1\ldots i_k})\Big|E_{k-1}\biggr)\leq(1-q)^{\delta\rho^{kN}}.
\end{equation}
We assume $E_{k-1}$, so, in particular, we know that
$B_{k-1}(z_{i_1\ldots i_{k-1}})$ holds.
That is, there exists a level-$g_{k-2}$ square $Q_{\mathrm{big}}$ such that
the event
$A_{kN}(Q_{\mathrm{big}},z_{i_1\ldots i_{k-1}})$ holds.
By definition, this means that we can find at least $[\delta\rho
^{kN}]+1$ squares in $Q_{\mathrm{big}}$ in $\mathcal{A}_{z_{i_1\ldots
i_{k-1}}}^{g_{k-1}}$ having center below, and at least as many squares
having center above, the line $e(z_{i_1\ldots i_{k-1}})$.
Using the second part of Lemma \ref{lem:B_k} (for $k$ instead of
$k+1$), we obtain that
the line $e(z_{i_1\dots i_k})$ has nice intersection with either all the
squares above or with all the squares below the line $e(z_{i_1\ldots i_{k-1}})$.
Without loss of generality, we may assume the former.

However, for all these squares $Q,$ the events $ A_{(k+1)N}^{c}
(Q,z_{i_1\ldots i_k})$ are (conditionally) independent, so, to
obtain (\ref{42}), it is enough to show that
\begin{equation}\label{43}
\mathbb{P}\bigl( A_{(k+1)N}^{c}(Q,z_{i_1\ldots i_k})| Q\in\mathcal{A}_{z_{i_1\ldots i_k}}^{g_{k-1}} \bigr)\leq1-q
\end{equation}
and this follows directly from equation (\ref{a}).
\end{pf}

\begin{lemma}\label{lem:product}
For all $n\geq1$,
we have
\begin{equation}\label{el}
\prod_{j=1}^{\infty}\bigl(1-4^{[2+\cdots+(j+1)]n}(1-q)^{\delta\rho^{jn}}\bigr)>0.
\end{equation}
\end{lemma}

\begin{pf}
We have to show that $\sum_{j=1}^\infty a_j$ converges, where
\[
a_j=4^{(1/2) j(j+1)n}(1-q)^{\delta\rho^{jn}}.
\]
It is therefore sufficient that $a_j\le e^{-j}$ for all large $j$.
This is true since

\[
\frac1j \log a_j= \frac12(j+1)n \log4 + \frac1{j}
\delta(\rho^n)^j \log(1-q) \le-1,
\]
which holds for $j$ large
enough since $\rho^n>1$ and $\log(1-q)<0$.
\end{pf}

Therefore, using Lemmas \ref{lem:subset}, \ref{lem:E_k} and \ref{lem:product}, we obtain that
\[
\pr(C_2-C_1\supset I)\geq q\prod_{k=1}^{\infty}\bigl(1-4^{[2+\cdots+(k+1)]N}(1-q)^{\delta\rho^{kN}}\bigr)>0.
\]
Combining this with Proposition \ref{01law} from Section \ref
{0--1-law}, this completes the proof of Theorem \ref{17}.

In the next six sections, we prove our \hyperref[15]{Main Lemma}.

\section{Distribution of types}\label{48}

In this section, the density function of $\Phi(Q,x)$ will be
determined for the four squares $Q$ from $\mathcal{S}_1$.

\subsection{The distribution of $\Phi(Q,x)$}

Let $U_1,U_2,U_3,U_4$ be four independent $\operatorname{Uniform}([0,\gap]
)$-distributed random variables.
The left corners of the two level-one intervals of the random Cantor
set $C_i$ are determined by $U_{2i-1},U_{2i}$ for $i=1,2$.
Let $(u_i,v_i)$ be the lower-left corner of the squares $Q_i$,
$i=1,\dots,4$ (see Figure \ref{fig:square}). Then,
\begin{eqnarray*}
(u_1,v_1) &=& \biggl(b+U_1,\frac{1}{2}+\frac{a}{2}+U_4\biggr), \\
(u_2,v_2) &=& \biggl(\frac{1}{2}+\frac{a}{2}+U_2,\frac{1}{2}+\frac{a}{2}+U_4\biggr),\\
(u_3,v_3) &=& \biggl(\frac{1}{2}+\frac{a}{2}+U_2,b+U_3\biggr),\\
(u_4,v_4) &=& (b+U_1,b+U_3).
\end{eqnarray*}

\begin{figure}

\includegraphics{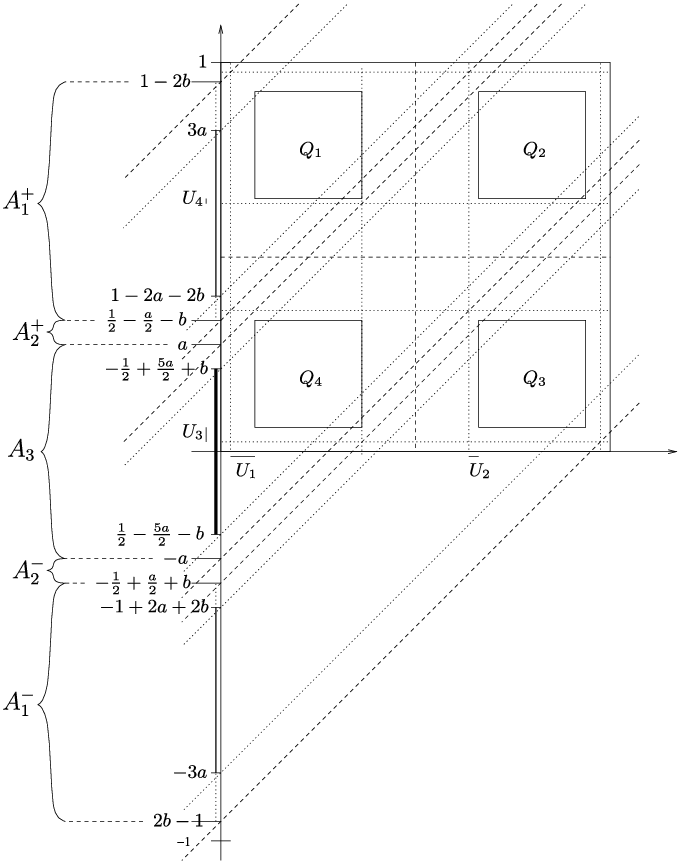}

\caption{If $x$ is an element of the bold vertical line,
then the line $e(x)$ intersects exactly two squares. If $x$ is an
element of one of the two plain vertical lines, then $e(x)$ intersects
one square. If $x$ is an element of one of the four dotted vertical
lines, then $e(x)$ intersects at most one square. If $x$ is such that
$a\leq x\leq 1-2a-2b$ or $-1+2a+2b\leq x\leq-a,$ then $e(x)$
intersects at most two squares with probability one. If $x$ is such
that $-\frac{1}{2}+\frac{5a}{2}+b\leq x\leq a$ or $-a\leq x\leq
\frac{1}{2}-\frac{5a}{2}-b,$ then $e(x)$ intersects exactly two squares.}\label{fig:square}
\end{figure}

For an $x\in[-1,1],$ we define $\Phi_i(x):=\Phi(Q_i,x)$.
From (\ref{eq:Phi}), simple computations yield
\begin{eqnarray}\label{eq:X(x)_type_def1}
\ds\Phi_1(x) &=& \cases{
\displaystyle\frac{1}{a}\biggl( -\frac{1}{2}-\frac{a}{2}+b+U_1-U_4+x\biggr), \cr
\hspace{37pt} \mbox{if } \displaystyle\frac{1}{a}\biggl(-\frac{1}{2}-\frac{a}{2}+b+U_1-U_4+x\biggr)\in[-1,1],\cr
\displaystyle\nex, \qquad \mbox{otherwise},}\nonumber\\ [-8pt]\\ [-8pt]
\ds\Phi_2(x) &=& \cases{\displaystyle\frac{1}{a}( U_2-U_4+x ), \qquad \mbox{if } \displaystyle\frac{1}{a}(U_2-U_4+x)\in[-1,1]\vspace{2pt},\cr
\displaystyle\nex, \hspace{86pt} \mbox{otherwise}\nonumber
}
\end{eqnarray}
and, similarly,
\begin{eqnarray}\label{eq:X(x)_type_def2}
\ds\Phi_3(x)  &=&  \cases{
\displaystyle\frac{1}{a}\biggl( \frac{1}{2}+\frac{a}{2}-b+U_2-U_3+x\biggr)\vspace*{2pt},\cr
\hspace{37pt}\mbox{if } \displaystyle \frac{1}{a}\biggl(\frac{1}{2}+\frac{a}{2}-b+U_2-U_3+ x\biggr)\in[-1,1],\cr
\displaystyle\nex, \qquad \mbox{otherwise},
}\nonumber\\ [-8pt]\\ [-8pt]
\ds\Phi_4(x)  &=&  \cases{
\displaystyle\frac{1}{a}( U_1-U_3+x ), \qquad \mbox{if } \displaystyle\frac{1}{a}(U_1-U_3+x )\in[-1,1]\vspace{2pt},\cr
\displaystyle\nex, \hspace{86pt}\mbox{otherwise}.\nonumber
}
\end{eqnarray}

To get a better geometric understanding of the distribution of the
random variables $\Phi_i(x),$ we define the three slanted stripes
$S_k$, $k=1,2,3$ (see Figure \ref{fig:simple_kernel}), in such a way
that $S_k\subset[-1,1]^2$ is bounded by the lines $\ell_{2k-1},\ell
_{2k}$, where
\begin{eqnarray}\label{eq:def_ell_1,2,3,4,5,6}
\hspace{29pt}\ds\ell_1 (x) &=& \frac{1}{a}x +\frac{1}{a}(1-a-2b), \qquad\ds\ell_2 (x) = \frac{1}{a}x +2,\qquad\ds\ell_3 (x) = \frac{1}{a}x
+\frac{\gap}{a},\nonumber\\ [-8pt]\\ [-8pt]
\hspace{29pt}\ds\ell_4 (x) &=& \frac{1}{a}x -\frac{\gap}{a},\qquad\ds\ell_5(x) = \frac{1}{a}x -2, \qquad\ds\ell_6 (x) = \frac{1}{a}x -\frac{1}{a}(1-a-2b).\nonumber
\end{eqnarray}

An immediate calculation shows that the following result holds.

%
\begin{figure}

\includegraphics{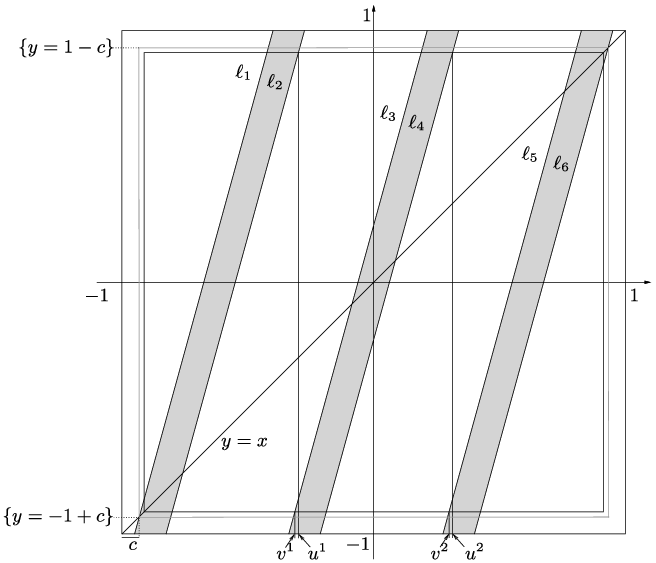}

\caption{The support of the density functions in the
simple case.}\label{fig:simple_kernel}
\end{figure}

\begin{lemma}\label{105}
For every $x\in[-1,1]$ and every $i=1,\dots,4,$ if $\Phi_i(x)\ne
\Theta, $ then
\[
(x,\Phi_i(x))\in S_1\cup S_2\cup S_3.\label{106}
\]
\end{lemma}

Let us call $\ell_j$ the graph of the function $\ell_j(x)$.
Observe that the reflection in the origin of $\ell_j$
is $\ell_{7-j}$ for $j=1,\dots,6$. For
a point $(x_1,x_2)\in\mathbb{R}^2$, we write $\pi_m(x_1,x_2):=x_m$,
$m=1,2$. We then define $c>0$ by
\[
-1+c:=\pi_1 (\ell_1\cap\{y=x\})\label{107}
\]
and obtain $c=\frac{2b}{1-a}$. By symmetry, it follows that
\[
1-c=\pi_1(\ell_6\cap\{y=x\}).\label{108}
\]
Using the fact that $-1+2b=\pi_1(\ell_1\cap\{y=-1\}),$
it follows from the symmetry mentioned above that
\begin{eqnarray}\label{109}
 && x\notin(-1+2b, 1-2b)\nonumber\\ [-8pt]\\ [-8pt]
 &&\qquad\Longrightarrow\quad e(x)\mbox{ does not intersect any level-one
 square}.\nonumber
\end{eqnarray}
The functions $\ell_1(x)$, $\ell_6(x)$ have repelling fixed point
$-1+c$, $1-c,$ respectively. Therefore,
\begin{eqnarray}\label{110}
&& x\in[-1,-1+c)\cup
(1-c,1]\nonumber\\ [-8pt]\\ [-8pt]
&&\qquad\Longrightarrow\quad \exists n \mbox{ such that }(x)\cap Q= \varnothing \mbox{ for all }
Q\in\mathcal{S}_n.\nonumber
\end{eqnarray}
With probability $1,$ no line $e(x)$ can intersect more than two
descendants, in fact, $[-1+2b, 1-2b]$ can be partitioned into five
sets, according to which descendants can be produced, given by (see
also Figure \ref{fig:square})
\begin{eqnarray}\label{def:A}
A_1^- &=& \biggl[-1+2b,-\dfrac{1}{2}+\dfrac{a}{2}+b\biggr),\qquad A_1^+=\biggl(\dfrac{1}{2}-\dfrac{a}{2}-b,1-2b
\biggr],\nonumber\\
A_2^- &=& \biggl[-\dfrac{1}{2}+\dfrac{a}{2}+b,-a\biggr),\qquad A_2^+=\biggl(a,\dfrac{1}{2}-\dfrac{a}{2}-b \biggr],\\
A_3 &=& \biggl[-a,a\biggr].\nonumber
\end{eqnarray}

\begin{lemma}\label{lem:at_most_two}
If $x\in A_3,$ then $x$ can only produce descendants with type $\Phi
_2(x)$ and/or $\Phi_4(x)$.
If $x\in A_1^+$ (resp. $x\in A_1^-$), then $x$ can produce at most one
descendant with type $\Phi_1(x)$ [resp. $\Phi_3(x)$].
If $x\in A_2^+,$ then there are two possibilities.
First, if $x$ produces $\Phi_1(x),$ then $\Phi_2(x)$ and $\Phi_4(x)$
cannot be born.
Second, if $x$ produces any of $\Phi_2(x)$ and $\Phi_4(x),$ then
$\Phi_1(x)$ cannot be born. If $x\in A_2^-$, then there are two
similar possibilities.
\end{lemma}
\begin{pf}
In Figure \ref{fig:square}, observe that $\operatorname{Proj}_{45^{\circ}}(Q_1)\cap
\operatorname{Proj}_{45^{\circ}}(Q_4)\ne\varnothing$ can happen only in the extreme
situation if the bottom of the square $Q_1$
is the same as the bottom of the dotted square which contains $Q_1$
on Figure \ref{fig:only_squares}. This means that $U_4=0,$ which
happens with probability zero. Similarly,
$\operatorname{Proj}_{45^{\circ}}(Q_3)\cap \operatorname{Proj}_{45^{\circ}}(Q_4)\ne\varnothing$
happens only if $U_2=0,$ which also has probability zero.
$\operatorname{Proj}_{45^{\circ}}(Q_1)\cap \operatorname{Proj}_{45^{\circ}}(Q_3)=\varnothing$
always holds, which completes the proof of our lemma.
\end{pf}

\subsection{The density functions}\label{sec:dens}
In this subsection, we will determine the density functions $f_{\Phi
_i(x)}(y)$ of the random variables $\Phi_i(x)$, $i=1,2,3,4,$ given
explicitly by (\ref{eq:X(x)_type_def1}) and (\ref{eq:X(x)_type_def2}).
We do not call them \textit{probability} density functions since the
$\Phi_i(x)$ may be equal to $\nex$ with
positive probability for some $x$.
The probability density function of the difference of two independent
$\operatorname{Uniform} ([0,\gap] )$-distributed random variables is the
triangular distribution given by $f_{\triangle}(z)=0$ if $|z|>\gap$ and
for $0\leq|z|\leq\gap$ by
\begin{equation}\label{eq:def_f}
f_{\triangle}(z)=\frac{1}{\gap^2}(\gap-|z|).
\end{equation}
To get $f_{\Phi_i(x)}(y),$ we apply simple transformations
to $f_{\triangle}(z)$ and find
\begin{eqnarray}\label{eq:f_density_function}
f_{\Phi_i(x)}(y) = af_{\triangle}(ay+c_i-x) \ind_{[-1,1]}(y)
\end{eqnarray}
with $c_1=-c_3=\frac{1}{2}+\frac{a}{2}-b$ and $c_2=c_4=0$.

From the definition,
\[
\pr\bigl(\Phi_i(x)=\nex\bigr)=1-\int_{[-1,1]} f_{\Phi_i(x)}(y)\,\di y.
\]

\section{A uniformly positive kernel}\label{sec:up_kernel_simple}

Here, and in the next two sections, we are going to define the type
space $T$ of the branching process introduced in Section~\ref{116}. In
order to ensure that conditions (\ref{cond:dens}), (\ref{cond:rho}), (\ref{cond:secmom}) of Section \ref
{117} hold, we introduce a type space $T$ which also satisfies
properties (A), (B), (C) of Section \ref{116}.
It follows from (\ref{110}) that we must choose our type space
$T\subset[-1+c,1-c]$.

Unfortunately, the construction of the type space $T$ satisfying the
above conditions is quite involved and technical for
those values of the parameters $a,b$ which do not satisfy (\ref{121}).
Therefore, we split the presentation into two parts.
In this section, we present the construction of $T$
across three lemmas: Lemmas \ref{lemA:support1}, \ref
{lemA:eigenvalue} and~\ref{lemA:support2}. In the next section, we
present the general case with the corresponding Lemmas \ref
{lem:support1}, \ref{lem:eigenvalue} and \ref{lem:support2}. The main
difference between these lemmas lies in the proofs of Lemmas \ref
{lem:support1} and \ref{lemA:support1}. Lemma \ref
{lem:eigenvalue} is almost the same as Lemma \ref{lemA:eigenvalue}.
Finally, the proof of Lemma \ref{lem:support2} follows the same line
as the proof of Lemma \ref{lemA:support2}, but is more technical.

\subsection{Descendant distributions and the kernel of the branching
process}\label{sec:6.1}

We introduce the random variables $X_1(x),X_2(x),X_3(x),X_4(x)$ for
$1\leq i\leq4$ by
\begin{equation}\label{def_X_i(x)}
X_i(x)=\cases{
\Phi_i(x), &\quad if $\Phi_i(x)\in T$,\cr
\nex, &\quad otherwise.
}
\end{equation}
So, the density of $X_i(x)$ is
\begin{equation}\label{def:f_xi}
f_{x,i}(y):=f_{\Phi_i(x)}(y)\ind_T(y)
\end{equation}
for $i=1,\dots,4$. In general, $X_i(x)$ also has an atom: $\pr
(X_i(x)=\nex)=1-\int_T f_{x,i}(y)\,\di y$.

\begin{figure}[b]

\includegraphics{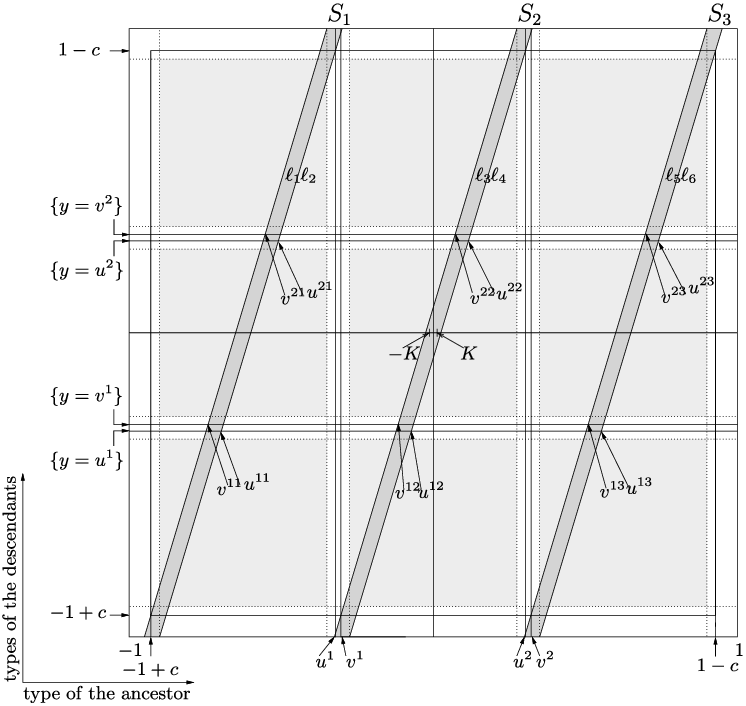}

\caption{Some points and lines related to the
kernel $m(x,y)$ if $l=1.$}\label{fig:kernel1}
\end{figure}

Recall [see equation (\ref{m-kernel})] that the kernel of the
branching process can be expressed as the sum of the density functions
of the random variables $X _i(x)$, $i=1,\dots,4$:
\[
m(x,y)=f_{x,1}(y)+f_{x,2}(y)+f_{x,3}(y)+f_{x,4}(y).\label{118}
\]
The structure of the support of this kernel is very important for the sequel.
Since the functions $f_{x,i}(y)$ ($i=1,2,3,4$) are piecewise continuous
on $[-1,1]$, $m(\cdot,\cdot)$ is piecewise continuous
on $[-1,1]\times[-1,1]$. The support of $m(\cdot,\cdot)$ is a subset
of the three slanting stripes $S_k$, $k=1,2,3,$ introduced earlier; see
also Figure \ref{fig:simple_kernel}.

\subsection{The possible holes in the support of the kernel of $\mathcal{Z}$}\label{123}

We have seen in (\ref{110}) that the branching process with ancestor
type in the set
$[-1,-1+c]$ or $[1-c,1]$ dies out in a finite number of generations
almost surely.
Therefore, it is reasonable to
restrict the type space to
$[-1+c+\eps,1-c-\eps]$ for some small positive $\eps$. However, in
some cases, we have to make further restrictions. Namely, for $i=1,2,$
we define
\begin{equation}\label{122}
u^i:=\pi_1(\ell_{2i}\cap\{y=1-c\}),\qquad v^i:=\pi_1(\ell_{2i+1}\cap\{y=-1+c\});
\end{equation}
see Figure \ref{fig:simple_kernel}.
Clearly, $u^1-v^1=u^2-v^2$ and an easy calculation shows that
\begin{equation}\label{111}
v^1<u^1 \quad\Longleftrightarrow \quad c<\frac{\gap}{2a}.
\end{equation}
We remark that this condition is equivalent to the condition in
equation (\ref{121}) (see also Figure \ref{fig:a-b-region}).
On the other hand, if $u^i<v^i$, $i=1,2,$ holds, then, for $x\in
[u^i,v^i],$ the set
\begin{equation}\label{113}
E_1(x):=\{y\dvtx m(x,y)>0\}
\end{equation}
is contained in $[-1,-1+c]\cup[1-c,1]$.
This implies that the process dies out in finitely many steps for $x\in
[u^i,v^i]$
(see Figure \ref{fig:kernel1}). Therefore, if the condition
stated in (\ref{111}) does not hold, then we have to make more
restrictions on our type space $[-1+c+\eps,1-c-\eps]$. This is what
we are going to do in Section \ref{sec:up_kernel}. For the convenience
of the reader, in Section \ref{112}, we treat the simpler case when
(\ref{111}) holds.

%

\section{A uniformly positive kernel in the simple case}\label{112}

In the remainder of this section, we will prove that if (\ref{111})
holds, that is, $v^1<u^1$, then we can choose a sufficiently small
$\varepsilon_0>0$ such that
\[
T=[-1+c+\varepsilon_0,1-c-\varepsilon_0]\label{114}
\]
satisfies conditions (\ref{cond:dens}), (\ref{cond:rho}) and (\ref
{cond:secmom}) [and also properties (A), (B), (C)]. The kernel in
the simple case is illustrated in Figure \ref{fig:simple_kernel}.

\setcounter{lemA}{6}
\renewcommand{\thelemA}{\arabic{lemA}A}
\begin{lemA}\label{lemA:support1}Assume that $v^1<u^1$. Fix an
$\varepsilon>0$ satisfying
\begin{equation}\label{eq:Aeps}
\eps<\frac{\gap}{2a}-c.
\end{equation}
Further, in this simpler case, let
\begin{equation}\label{eq:AT}
\T=\T(\eps)=[-1+c+\eps,1-c-\eps].
\end{equation}
Then, the kernel $m(x,y)$ of the branching process $\mathcal{Z}$ has
the following property:
\begin{equation}\label{condA:bigger_kappa}
\quad\exists\kappa>0 \mbox{ such that } \forall x\in T, \mbox{ the set } E_1(x) \mbox{ contains an interval of length }\kappa.
\end{equation}
\end{lemA}

\begin{pf}
There are two possibilities for the shape of $E_1(x)$ [defined in (\ref{113})]:
\begin{longlist}[(1)]
\item[(1)]$E_1(x)$ \textit{consists of two intervals}:
$[-1+c+\eps,\ell_{2k+1}(x))\cup(\ell_{2k}(x),1-c-\eps]$
(for\vspace*{1.5pt} $k=1$ or $k=2$). The length of one of these intervals is at least
half of $\ell_3(u^1)-(-1+c+\eps)$, that is, $\kappa_1=\frac
{1}{2}\cdot(\frac{\gap}{a}-2c) $.

\item[(2)]$E_1(x)=(\ell_{2k-1}(x),\ell_{2k}(x))$ (\textit{for some }$1\leq
k\leq3$)  \textit{is an open interval with length} $\kappa_2=\frac{4}{a}\gap$.
\end{longlist}
Summarizing these cases, define $\kappa=\min\{\kappa_1,\kappa_2\}$.
\end{pf}

\begin{lemA}\label{lemA:eigenvalue}
Let $m^\eps$ be the kernel in Lemma \ref{lemA:support1} with
type space $\T=\T(\eps),$ as in {(\ref{eq:AT})}. One can choose
$\eps>0$ which satisfies {(\ref{eq:Aeps})} such that the
largest eigenvalue of $m^\eps$ is larger than 1. From now on, we fix
such an $\varepsilon$ and call it $\varepsilon_0$.
\end{lemA}

\begin{pf} Let $T(0):=[-1+c,1-c]$, with corresponding kernel $m^0$.
Define [as in (\ref{10})] the operator $\mathcal{T}_\eps$ for all
$\eps\ge0$ by
\[
\mathcal{T}_\eps h(y)=\int_{\mathbb{R}} h(x) m^\eps(x,y)\,\di x
\]
for functions with $\operatorname{supp}(h)\subset{\T}(\eps)$.

We shall prove that $4a$ is an eigenvalue of the operator $\mathcal
{T}_0$ with eigenfunction $h(x) =\ind_{T(0)}(x)$:
\begin{eqnarray*}
\mathcal{T}_0 h(y)&=&\int_{\mathbb{R}}h(x)m^0(x,y)\,\di x\\
&=&\int_{\mathbb{R}}h(x)\Biggl(\sum_{i=1}^4f_{x,i}(y)\Biggr)\ind_{T(0)}(y)\,\di x\\
&=&4ah(y)\int_{T(0)}\sum_{i=1}^4f_{\triangle}(ay+c_i-x)\,\di x\\
&=&4ah(y),
\end{eqnarray*}
provided we show that for all $i=1,2,3,4,$
\[
\int_{[-1+c,1-c]}f_{\triangle}(ay+c_i-x)\,\di x=1.
\]
Since $f_{\triangle}$ is a probability density with support lying in
$[-\gap,\gap]$, it then suffices to show that
for all $y\in[-1+c,1-c]$ and for $i=1,2,3,4$, we have
\[
ay+c_i-1+c\le-\gap\quad\mbox{and}\quad ay+c_i+1-c\ge\gap.
\]
Taking the worst case for $y,$ this boils down to showing
\[
a(1-c)+c_i-1+c\le-\gap\quad\mbox{and}\quad a(-1+c)+c_i+1-c\ge\gap.
\]
For $i=1,$ we have $c_1=(a+1)/2-b$, so there we have to check that
\[
(1-c)(a-1)+\frac{a+1}{2}-b\le-\gap\quad\mbox{and}\quad(1-c)(1-a)+\frac{a+1}{2}-b\ge\gap.
\]
First, note that since $c_3=-c_1$, the case $i=3$ is covered by the
case $i=1$.
Further, note that the left inequality implies the right one since
$a+1> 2b$ always holds.
Moreover, $a+1> 2b$ also gives that the left inequality will imply both
inequalities for $i=2,4$.
The calculation is then completed by substituting $c=2b/(1-a)$ in the
left inequality, which turns out to be an equality.

The conclusion of the lemma follows from a simple fact noted by Larsson
\cite{Larsson}: if the two kernels $m^0$ and $m^\eps$ are close to
each other in
$L^2$-sense, then the eigenvalues of the operators $\mathcal{T}_0$ and
$\mathcal{T}_\eps$ are close to each other.
\end{pf}

\begin{lemA}\label{lemA:support2}
Let $\T$ be as in Lemma \ref{lemA:eigenvalue}.
Then there exists an index $n$ such that for all $x\in\T$, $\{
y\dvtx m_n(x,y)>0\}=\T$.
\end{lemA}

Since the function $m_n(\cdot,\cdot)$ is piecewise continuous on the
compact set $T$, Lemma \ref{lemA:support2} implies that there exists
an $a_{\min}>0$ such that $m(x,y)\geq a_{\min}$ for any $x,y\in T$.
Further, using the fact that $m(x,\cdot)$ is bounded, we immediately
obtain that $a_{\max}:=\sup_{x\in T}\ev_x \mathcal{Z}^2_1(T)$ is
finite. Therefore, we have the following result.
\begin{corollary}
Let $\T$ be as in Lemma \ref{lemA:eigenvalue}. The branching process
$\mathcal{Z}$ with type space $T$ satisfies conditions (\ref{cond:dens}) and (\ref{cond:secmom}).
\end{corollary}

\begin{pf*}{Proof of Lemma \ref{lemA:support2}}
Basically, we will prove that if (\ref{condA:bigger_kappa}) holds,
then Lemma \ref{lemA:support2} also holds since the slope of the lines
$\ell_i$ is equal to $\frac{1}{a},$ which is bigger than one.
Let $E_n(x)=\{y\dvtx m_n(x,y)>0\}$.
We will prove that in both cases of the proof of Lemma \ref
{lemA:support1}, the sequence $(E_n(x))$ reaches the whole type space
in a finite number of steps, uniformly in $n$ and $x\in\T$.

We can derive $E_{n+1}(x)$ from $E_n(x)$ by means of the equation
\[
m_{n+1}(x,y)=\int_{\T}m_n(x,z)m_1(z,y)\,\di z,
\]
which implies that
\begin{equation}\label{eq:E_{n+1}_decompA}
E_{n+1}(x)=\bigcup_{y\in E_n(x)}E_1(y).
\end{equation}
In the proof of Lemma \ref{lemA:support1}, we treated two separate
cases. We continue this proof according to those two cases:
\begin{longlist}[(1)]
\item[(1)]\textit{$E_1(x)$ consists of two intervals.} Take the longer one,
so its length is at least $\kappa_1=\frac{1}{2}\cdot(\frac
{\gap}{4a}-2c)$. The following two facts hold. This interval
contains either $-1+c+\eps$ or $1-c-\eps,$ and if $E_n(x)$ contains
one of these points, then $E_{n+1}(x)$ also contains the same point
because of (\ref{eq:E_{n+1}_decompA}). Therefore, if $E_n(x)\neq\T$
and is of the form, for example, $[-1+c+\eps,-1+c+\eps+s)$ for some
positive $s,$ then $E_{n+1}(x)\supset[-1+c+\eps,-1+c+\eps
+\frac{1}{a}s)$ or $E_{n+1}(x)=\T$. Hence, if
$E_1(x)=[-1+c+\eps,-1+c+\eps+s),$ then in
\[\label{eq:n1b}
n_1(x)=\biggl\lceil\log_{1/a}\biggl(\frac{2(1-c-\eps)}{s}\biggr)\biggr\rceil
\]
steps, $E_n(x)$ reaches $\T$, that is, $E_{n_1(x)}(x)=\T$. $s\geq
\kappa_1$ implies that $n_1(x)\leq\lceil\log_{1/a}(\frac{2(1-c-\eps)}{\kappa_1} )\rceil=n_1^*$.

\item[(2)]\textit{$E_1(x)=(\ell_{2k-1}(x),\ell_{2k}(x))$ \textup{(}for some $1\leq
k\leq3$\textup{)} is an open interval with length} $\kappa_2=\frac{4}{a}\gap
$. If, for some $n,$ $E_n(x)$ does not contain either $-1+c+\eps$ or
$1-c-\eps,$ then we have three possibilities for $E_{n+1}(x)$: (i) it
does not contain any of these two points; (ii) it contains one of
them; (iii) it equals $\T$. In case (iii) we obtained what we wanted. In case (i), the length of $E_{n+1}(x)$
equals $\frac{1}{a}|E_n(x)|+\frac{2\gap}{a}$; in case (ii), we have
$E_{n+n_1^*}(x)=\T$ by (1) above, so we estimate the number of
necessary iterations from below if we suppose that case (i) happens in
each step then case (ii) in $n_1^*$ number of steps. As in (\ref{P}), we have a uniform bound for the number of iterations
in (\ref{16}): $n_2^*=\lceil\log
_{1/a}(\frac{2(1-c-\eps)}{\kappa_2} )
\rceil$. Therefore, in this case, we have $E_{n_1^*+n_2^*}(x)=\T$ for
any $x$.
\end{longlist}
Summarizing these considerations, one obtains that for $n\geq
n_1^*+n_2^*$, one has $E_n(x)=\T$.
\end{pf*}

\section{A uniformly positive kernel in the general case}\label{sec:up_kernel}

The construction of $T$ consists of two steps. We will call any open
subset of $[-1,1]$ a \textit{pre-type space}. First, we inductively
construct a sequence of pre-type spaces $T^0\supset T^1\supset\cdots
\supset T^l$ and prove that $T^r$, $r=0,\dots,l,$ consists of $3^r$
disjoint open intervals of equal length. Those elements of $T^l$ which
are ``far'' from the endpoints of the components of $T^l$ satisfy (\ref
{cond:bigger_kappa}). Unfortunately, the same does not hold for the
points close to the the boundary of the components of $T^l$. So, as a
second step of the construction of $T,$ we remove a small neighborhood
of the boundary of $T^l$ from $T^l$.

\begin{lemma}\label{lem:support1}
There exists a restriction of the pre-type space $(-1+c,1-c)$ to a
closed set $\T$ such that the kernel $m$ of the branching process
$\mathcal{Z}$ with type space $\T$ satisfies
\begin{equation}\label{cond:bigger_kappa}
\qquad\exists\kappa>0 \mbox{ such that } \forall x\in T,\mbox{ the set } E_1(x)\mbox{ contains an interval of length }\kappa.
\end{equation}
%
%
%

%
Further,
$\T$ consists of $3^l$ disjoint
closed intervals of equal length for some $l\in\mathbb{N}$. Moreover,
$0$ is contained in the interior of $\T$.
\end{lemma}
\begin{pf}
We recall that $u^1,v^1$ were defined in (\ref{122}) and we take the
pre-type space ${\T}^0:=(-1+c,1-c)$.
If $v^k<u^k$, then we define $l:=0$ and the proof
of (\ref{cond:bigger_kappa}) was achieved in Lemma \ref{lemA:support1}.
So, we can assume that $u^k\leq v^k$, $k=1,2$.
To ensure that (\ref{cond:bigger_kappa}) holds, we need to remove the
intervals $[u^1,v^1]$ and $[u^{2},v^{2}]$ from the
pre-type space $T^0$ (see Figure \ref{fig:kernel1}). So, we restrict
ourselves to the next pre-type space: ${\T}^1={\T}^0\setminus\{
[u^1,v^1] \cup[u^2,v^2] \}$.
The size of each of the intervals removed is
$\varrho_1:=v^1-u^1=v^2-u^2 $.
We define the second generation endpoints $u^{i_1k}$ and $v^{i_1k}$ as follows:
\[
u^{i_1k}=\pi_1(\{y=u^{i_1} \}\cap\ell_{2k})\quad\mbox{and}\quad v^{i_1k}=\pi_1(\{y=v^{i_1} \}\cap\ell_{2k-1}),
\]
where $i_1=1,2$ and $k=1,2,3$; see Figure \ref{fig:kernel1}. If
$v^{i_1k}<u^{i_1k}$, then we define $l:=1$. Otherwise,
we continue defining the sets
${\T}^r$ and the endpoints of the subtracted intervals $v^{i_1\dots
i_r}$ and $u^{i_1\dots i_r}$ ($i_1=1,2$, $i_2,\dots,i_r=1,2,3$) as
follows: assuming that $u^{i_1\dots i_{r-1}}\leq v^{i_1\dots i_{r-1}}$,
we define the level-$r$ endpoints as
\begin{eqnarray}\label{128}
u^{i_1\dots i_{r-1}k}&=&\pi_1(\{y=u^{i_1\dots i_{r-1}}
\}\cap\ell_{2k})\quad\mbox{and}\nonumber\\ [-8pt]\\ [-8pt]
v^{i_1\dots i_{r-1}k}&=&\pi_1(\{y=v^{i_1\dots i_{r-1}}
\}\cap\ell_{2k-1})\nonumber
\end{eqnarray}
for $i_1=1,2$ and $i_2,\dots,i_{r-1},k=1,2,3$. Put
\begin{equation}\label{300}
{\T}_{r}={\T}_{r-1}\setminus\{[u^{i_1i_2\dots
i_{r}},v^{i_1i_2\dots i_{r}} ], i_1=1,2,i_2,\dots,i_{r}=1,2,3
\}.
\end{equation}
The size of each of the intervals removed is
$\varrho_{r}:=v^{i_1i_2\dots i_{r}}-u^{i_1i_2\dots i_{r}}$.
Using $\ell_{2k}(x)-\ell_{2k-1}(x)=2\gap/a$ (see also the left-hand
side of Figure \ref{fig:rho_recursion}),\vadjust{\goodbreak} one can easily check that
\begin{equation}\label{124}
\forall r\geq1,\qquad\rho_{r+1}=a\rho_{r}-2\gap\quad\mbox{and}\quad\rho_1=v^1-u^1.
\end{equation}
Consider the smallest $r\geq1$ for which
$v^{i_1\dots i_{r+1}}<u^{i_1\dots i_{r+1}}$ or, equivalently, \mbox{$\rho_{r+1}<0$}. We then set $l=r$ and the recursion ends. The fact that $l$
is finite is immediate from (\ref{124}).

\begin{figure}

\includegraphics{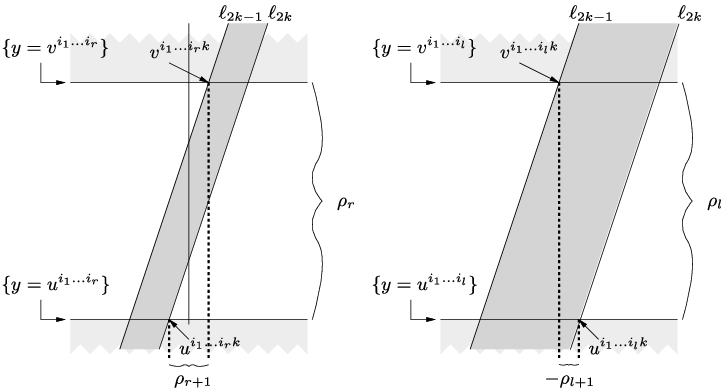}

\caption{The recursion of $\{\rho_r\}_r$. On
the left-hand side, $r\leq l-1$.}\label{fig:rho_recursion}
\end{figure}

We can represent $T^{l-1}$ and $T^{l}$ as follows:
\[
T^{l-1}=\bigcup_{j=1}^{3^{l-1}}(\gamma _j,\delta _j),\qquad T^l=\bigcup_{i=1}^{3^l}(\alpha_i,\theta_i).\label{125}
\]
Using (\ref{128}),
it follows from elementary geometry (see Figure \ref{fig:rho_recursion}) that
\begin{eqnarray}\label{129}
\forall i, \exists j,\exists k\mbox{:}\qquad\alpha_i &=&\pi_1\bigl(\{(x,y)\dvtx
y=\gamma_j\}\cap\ell_{2k-1}\bigr),\nonumber\\ [-8pt]\\ [-8pt]
\theta _i&=&\pi_1\bigl(\{(x,y)\dvtx y=\delta
_j\}\cap\ell_{2k}\bigr).\nonumber
\end{eqnarray}

We need further restrictions because condition (\ref{cond:bigger_kappa})
is not satisfied around the endpoints
$\alpha_i,\beta_i$. Therefore, we remove sufficiently small intervals
from both ends of each of the $3^l$ intervals of $T^l$. Namely, we
define the type space of the process by
\begin{equation}\label{eq:truncation}
T(\varepsilon):=\bigcup_{i=1}^{3^l}[\alpha_i+\varepsilon,\beta
_i-\varepsilon],
\end{equation}
where
\begin{equation}\label{127}
0<\varepsilon<\frac{\gap}{a} - \frac{1}{2}\rho_l.
\end{equation}

This bound will be used in part (c) at the end of this proof. For any
$j\in\{1,\dots,3^{l-1}\},$
we can find $i'\in\{1,\dots,3^l\}$
such that
\begin{equation}\label{130}
[ \gamma_j+\varepsilon,\delta_j-\varepsilon]=
\bigcup_{m=0}^{2}
[\alpha_{i'+m}+\varepsilon,\beta_{i' +m}-\varepsilon]
\cup \bigcup_{h=1}^{2} R^{(j)}_h,
\end{equation}
where $R^{(j)}_h$, $h=1,2,$ are intervals of length $\rho
_{l}+2\varepsilon$; see Figure \ref{fig:rho_recursion}.
Further, for every $1\leq i\leq3^l,1\leq j\leq3^{l-1},$ the set
$(\alpha_i+\varepsilon,\beta_i-\varepsilon)\times(\gamma
_j+\varepsilon,\delta _j-\varepsilon)\cap T(\eps)\times T(\eps)$
consists of three congruent squares aligned on top of each other,
of side-length
\[
s:=\beta_i-\alpha_i-2\varepsilon.
\]
The distance between two neighboring squares is $\rho_l+2\eps$.
%

We now prove that (\ref{cond:bigger_kappa}) holds. That is, we want to
estimate the length of the longest interval in $E_1(x)$ from below. The
argument uses only elementary geometry.

For any $x\in T(\varepsilon),$ there is a unique $k\in\{
1,2,3\}$ such that $E_1(x)\subseteq(\ell_{2k}(x),\break\ell
_{2k-1}(x))$ holds.
Using (\ref{eq:def_ell_1,2,3,4,5,6}), one can immediately see that
the length of the interval $(\ell_{2k}(x),\ell_{2k-1}(x))$ is $\frac
{2\gap}{a}$. Geometrically, this means that the vertical line through
$x$ intersects the stripe $S_k$ in a (vertical) interval of length
$\frac{2\gap}{a}$.

Since there are many holes in $T(\varepsilon)$, for some $x\in
T(\varepsilon)$, the set $E_1(x)$ consists of at most three
subintervals of $(\ell_{2k}(x),\ell_{2k-1}(x))$; see Figure \ref{fig:rho_recursion}.
We prove that the maximum length of these intervals is uniformly
bounded away from zero.

Fix
a component $[\alpha_i+\eps,\beta_i-\eps]\subset T(\varepsilon)$
and let $x\in[\alpha_i+\eps,\beta_i-\eps]$. For this~$i,$ we
choose $j$ and $k$
according to the formula (\ref{129}).
We now distinguish three possibilities for $x\in T(\eps)$:
\begin{longlist}
\item[(a)] first we assume that the intersection of the vertical line
through $x$ with the stripe $S_k$
is not contained in the rectangle
$[\alpha_i+\varepsilon,\beta_i-\varepsilon] \times
[\gamma_j+\varepsilon,\delta_j-\varepsilon]$ [see
Figure \ref{fig:rho_recursion}], then, using the fact that the slope of the
lines $\ell_m$, $m=1,\dots,6,$ is $1/a>3,$ by elementary geometry, we
obtain that the set
$E_1(x)$ contains an interval of length $\kappa:=\frac{1}{a}\eps
-\eps>2\eps>0$
(see Figure \ref{fig:rho_recursion}B);

\item[(b)] next, we assume that there exists $m\in\{0,1,2
\}$ such that the intersection of the vertical line through $x$ with
the stripe $S_k$
is contained in the square $[\alpha_i+\varepsilon,\beta
_i-\varepsilon]\times[\alpha_{i'+m}+\varepsilon,\beta
_{i' +m}-\varepsilon]$, where $i'$ is defined as in (\ref
{130})---in this case, the set $E_1(x)=(\ell_{2k}(x),\ell_{2k-1}(x))$
and then the assertion holds with the choice of $\kappa:=\frac{2\gap
}{a}>0$ [see (\ref{127})];

\item[(c)] finally, we assume that the intersection of the vertical
line through $x$ with the stripe $S_k$ has a nonempty intersection with
one of the rectangles $[\alpha_i+\varepsilon,\beta
_i-\varepsilon]\times R^{(j)}_h$, $h=1,2$---in this case, by
elementary\vadjust{\goodbreak} geometry (see Figure \ref{fig:rho_recursion}A), $E_1(x)$ contains an
interval of length at least
\begin{eqnarray*}
\kappa &:=&\min\biggl\{s,\frac{1}{2}\cdot\bigl(\ell_{2k-1}(x)-\ell
_{2k}(x)\bigr)-(\rho_{l}+2\varepsilon)\biggr\}\\
 &=&\min\biggl\{s,\frac{1}{2}\biggl(\frac{2\gap}{a}-(\rho_{l}+2\varepsilon)\biggr)\biggr\}.
\end{eqnarray*}
\end{longlist}
It follows from (\ref{127}) that $\kappa>0$.
\end{pf}

\begin{figure}

\includegraphics{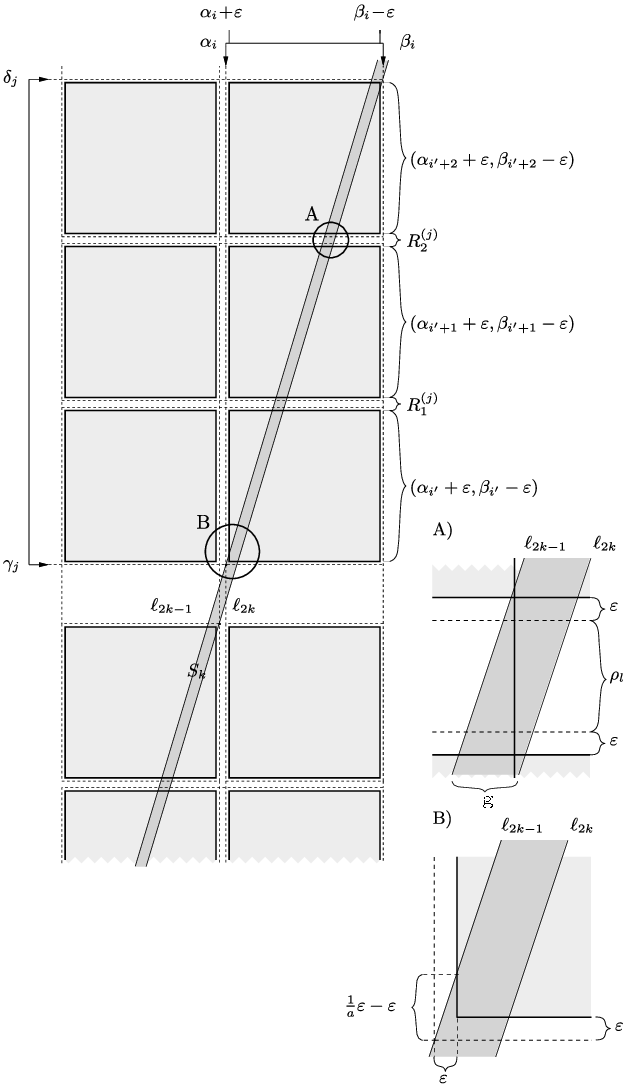}

\caption{Stripe $S_k$ and level-$l$ squares.}\label{fig:130}
\end{figure}

We will now deal with the problem of still having a kernel with
largest eigenvalue larger than 1.

\begin{lemma}\label{lem:eigenvalue}
Let $m^\eps$ be the kernel in Lemma \ref{lem:support1} with
type space $\T=\T(\eps)$. One can choose $\eps$ so small that the
largest eigenvalue of $m^\eps$ is larger than~1.
\end{lemma}

\begin{pf}
Changing $\T^0$ to $\T^l$ in the proof of Lemma \ref{lemA:eigenvalue}, we obtain the proof of Lemma \ref{lem:eigenvalue}.
More precisely, it is enough to prove that $4a$ is an eigenvalue of the
operator $\mathcal{T}_l$ with eigenfunction $h(x) =\ind_{T^l}(x),$
where $T^l$ is defined in the proof of Lemma \ref{lem:support1}:
\begin{eqnarray*}
\mathcal{T}_l h(y)&=&\int_{\mathbb{R}}h(x)m(x,y)\,\di x\\
&=&\int_{\mathbb{R}}h(x)\Biggl(\sum_{i=1}^4f_{x,i}(y)\Biggr)\ind
_{T^l}(y)\,\di x\\
&=&4ah(y)\int_{T^l}\sum_{i=1}^4f_{\triangle}(ay+c_i-x)\,\di x\\
&=&4ah(y),
\end{eqnarray*}
provided we show that for all $i=1,2,3,4$ and for all $y\in T^l,$
\[
\int_{T^l}f_{\triangle}(ay+c_i-x)\,\di x=1.
\]
So, we have to show that
for all $y\in T^l$ and for $i=1,2,3,4$, we have
\begin{equation}\label{301}
\{x\dvtx f_{\triangle}(ay+c_i-x)>0\}\subset T^l.
\end{equation}
This holds since we have constructed the intermediate type space $T^l$
so that this property is satisfied; see the left figure in Figure \ref{fig:130}. We have subtracted intervals of the form
$(u^{i_1\dots i_r k},v^{i_1\dots i_r k})$ in (\ref{300}) during the
construction of successive intermediate type spaces $T^{r+1}$,
$r=0,\dots,l-1$. If $y\in T^{r+1}$, then each interval\vspace{1pt} of the form
$(u^{i_1\dots i_r k},v^{i_1\dots i_r k})$ is disjoint from $[\ell
_{2k-1}^{-1}(y),\ell_{2k}^{-1}(y)]$ for all $y\in T^l$ and $k=1,2,3$.
Therefore, for any $y\in T^l$, we have $[\ell_{2k-1}^{-1}(y),\ell
_{2k}^{-1}(y)]\subset T^l$. Further, for any $i=1,2,3,4,$ there exists
a positive integer $k_i$ ($k_1=1$, $k_2=k_4=2$, $k_3=3$) such that
\[
\{x:f_{\triangle}(ay+c_i-x)>0\}=(\ell_{2k_i-1}^{-1}(y),\ell_{2k_i}^{-1}(y)).
\]
Hence, (\ref{301}) holds.

The proof is now completed analogously to the proof of Lemma
\ref{lemA:eigenvalue}.
\end{pf}

%

\begin{lemma}\label{lem:support2}
Let $\T$ be as in Lemma \ref{lem:eigenvalue}.
There then exists an $n$ such that for all $x\in\T$, $\{y:m_n(x,y)>0\}
=\T$.
\end{lemma}

Since the function $m_n(\cdot,\cdot)$ is piecewise continuous on the
compact set $T$, Lemma \ref{lem:support2} implies that there exists an
$a_{\min}>0$ such that $m(x,y)\geq a_{\min}$ for any $x,y\in
T$. Further, using the fact that $m(x,\cdot)$ is bounded, we
immediately obtain that $a_{\max}:=\sup_{x\in T}\ev_x [\mathcal
{Z}^2_1(T)]$ is finite. Therefore, we have
the following result.
\begin{corollary}
Let $\T$ be as in Lemma \ref{lem:eigenvalue}. The branching process
$\mathcal{Z}$ with type space $T$ satisfies the conditions (\ref{cond:dens}) and (\ref{cond:secmom}).
\end{corollary}

\begin{pf*}{Proof of Lemma \ref{lem:support2}}
We will prove the lemma in two steps. Recall the definition of
$E_n(x)$: $E_n(x)=\{y\dvtx m_n(x,y)>0\}$.
\begin{step}\label{step1}
$\forall x\in T, \exists i,n$ such that $[\alpha_i+\eps,\beta_i-\eps]\subset E_n(x)$
implies that $E_{n+l}(x)=T$.
\end{step}
\begin{step} \label{step2}
There exists an $N$ such that for every
$x\in T,$ we can find a positive integer $n(x)\leq N$
such that the following holds:
\[
\exists i,\qquad[\alpha_i+\eps,\beta_i-\eps]\subset
E_{n(x)}(x).
\]
\end{step}
As a corollary of these two statements, we obtain that the assertion of
the lemma holds with the choice $n=N+l$. Namely, for any $x\in T$, we
have $E_{N+l}(x)=T$.

\begin{pf*}{Proof of Step \ref{step1}} To verify Step \ref{step1}, we first observe that by
(\ref{eq:E_{n+1}_decompA}), we have
\begin{eqnarray}\label{eq:E_{n+1}_decomp}
\qquad E_{n+1}(x)&=&\bigcup_{y\in E_n(x)}E_1(y)\nonumber\\ [-8pt]\\ [-8pt]
          &=&\bigcup_{y\in E_n(x)}\bigl((\ell_2(y),\ell_1(y))\cup(\ell_4(y),\ell_3(y))\cup(\ell_6(y),\ell_5(y))
          \bigr)\cap\T.\nonumber
\end{eqnarray}
Fix an $i\in\{1,\dots,3^{l}\}$.
First, we define $\alpha_{i,l-r}$ and $\beta_{i,l-r}$ for $r=0,\dots
,l,$ inductively.
For $r=0,$ let $(\alpha_{i,l},\beta_{i,l}):=(\alpha_i,\beta_i)$.
Assume that we have already defined $(\alpha_{i,l-r},\beta_{i,l-r})$.
Using (\ref{128}), we define
$
\alpha_{i,l-(r+1)}$ and $\beta_{i,l-(r+1)}
$ as the unique numbers satisfying
\begin{eqnarray}\label{133}
\alpha_{i,l-r}&=&\pi_1\bigl(\bigl\{(x,y)\dvtx
y=\alpha_{i,l-(r+1)}\bigr\}\cap\ell_{2k(r)-1}\bigr),\nonumber\\ [-8pt]\\ [-8pt]
\beta_{i,l-r}&=&\pi_1\bigl(\bigl\{(x,y)\dvtx
y=\beta_{i,l-(r+1)}\bigr\}\cap\ell_{2k(r)}\bigr),\nonumber
\end{eqnarray}
where $k(r)=1,2,3$. Then, by the construction, we have $(\alpha
_{i,0},\beta_{i,0})=(-1+c,1-c)$.
Let $x\in T$. According to the assumption of Step \ref{step1},
we can find $i,n$ such that
\begin{equation}\label{131}
[\alpha_i+\eps,\beta_i-\eps]=(\alpha_i,\beta_i)\cap T\subset E_n(x)
\end{equation}
holds.
Using induction, we prove that
\begin{equation}\label{132}
E_{n+r}(x)\supset(\alpha_{i,l-r},\beta_{i,l-r})\cap T \qquad\mbox{for }0\leq r\leq l.
\end{equation}
Namely, for $r=0,$ the assertion in the induction is identical to (\ref{131}).
We now suppose that (\ref{132}) holds for $r< l$.
By (\ref{eq:E_{n+1}_decomp}) and (\ref{133}), we have
\begin{eqnarray*}
E_{n+r+1}(x)&=&\bigcup_{y\in E_{n+r}(x)}E_1(y)\\
&\supset& \bigcup_{y\in(\alpha_{i,l-r},\beta_{i,l-r})\cap T} \bigl(\ell_{2k(r)}(y),\ell_{2k(r)-1}(y)\bigr)\cap T\\
&=&\bigl(\alpha_{i,l-(r+1)},\beta_{i,l-(r+1)}\bigr)\cap T,
\end{eqnarray*}
which completes the proof of (\ref{132}). We apply (\ref{132})
for $r=l$. This yields
that $E_{n+l}=(-1+c,1-c)\cap T =T$ holds.
\end{pf*}

\begin{pf*}{Proof of Step \ref{step2}} First, observe that the largest interval in
$E_1(x)$ either has an endpoint that is an endpoint of a connected
component of $T$ [this happens in case (a) and (c) in the end of the
proof of Lemma \ref{lem:support1}] or $E_1(x)=(\ell_{2k_1}(x),\ell
_{2k_1-1}(x))$ [which is case (b) in the same proof]. However, in the
last case, using (\ref{eq:E_{n+1}_decomp}), after $N_1$ steps, where
$N_1$ is the smallest solution of the inequality $(\frac
{2}{a})^{N_1}\cdot\frac{2\gap}{a}>s$, we obtain that the
largest interval contained in $E_{N_1}(x)$ has an endpoint of a
connected component of $T$ (see Figure \ref{fig:rho_recursion}) and its length
is greater than $\kappa$.
In this way, because of the symmetry between the endpoints of the
connected components of $T,$ from now on, we may assume that
$[\alpha_i+\eps,\alpha_i+\eps+z_1)\subset E_1(x),$ where $z_1\geq
\kappa$.
Using (\ref{eq:E_{n+1}_decomp}), we can write
\begin{eqnarray}\label{134}
E_2(x)&\supset& \bigcup_{y\in[\alpha_i+\eps,\alpha_i+\eps+z_1)}(\ell_{2k_1}(y),\ell_{2k_1-1}(y))\cap T\nonumber\\
&=&\bigl(\ell_{2k_1}(\alpha_i+\eps),\ell_{2k_1-1}(\alpha_i+\eps+z_1)\bigr)\cap T\\
&=&\bigl[\alpha^{(2)}+\eps,\alpha^{(2)}+\eps+z_2\bigr)\cap T\nonumber
\end{eqnarray}
for some $k_1\in\{1,2,3\}$, left endpoint
$\alpha^{(2)}\in T$ and $z_2>\frac{1}{a}z_1\geq\frac{1}{a}\kappa$.
If $z_2<s,$ then the largest connected component of $E_2(x)$ has a left
endpoint of one of the connected components of $T$, $\alpha^{(2)}$,
but the other endpoint is in the interior of the same connected
component of $T$. If $z_2\geq s,$ then $E_2(x)$ clearly contains
a 
connected component of $T$. For $E_n(x)$, $n\geq3,$ we can inductively
define $k_n$, left endpoint $\alpha^{(n)}$ and length $z_n$ in the
same way as above. Observe that $z_n>(\frac{1}{a})^{n-1}
\kappa$ for any $n\geq2$. Let $N_2$ the smallest solution of the
inequality $(\frac{1}{a})^{N_2-1}\kappa>s$. Then,
$E_{N_2}(x)$ contains a connected component of $T$.

Let $N=N_1+N_2$. Then, $E_N(x)$ contains a connected component of $T$.
\end{pf*}
\noqed
\end{pf*}

\section{Uniform exponential growth}\label{UEG}
In this section, we want to prove an extension of Theorem \ref{thmA:14.1} stating that the population can grow uniformly
exponentially starting from any element of a special interval. For the
precise statement, see Lemma \ref{lem:p2_uniformly_positive}.

First, we will determine the density of the measure $\pr_x (\Z
_1(A)\in\cdot)$.
We use the notation of Lemma \ref{lem:at_most_two} and define, for
$x_1,x_2\in T,$
\[
\pr_{x_1,x_2}:=\pr_{x_1}\otimes\pr_{x_2},
\]
the convolution of the measures $\pr_{x_1}$ and $\pr_{x_2}$.
Recalling the definitions of $A_1^+$, $A_2^+$, $A_3$, $A_2^-$, $A_1^-$
in equation (\ref{def:A}) and the definition of $f_{x,i}$,
$i=1,2,3,4$, in equation (\ref{def:f_xi}), we can state the following lemma.

\begin{lemma}\label{lem:decomp_in_the_first_generation}
For $x\in\T, A\subset T$ and a natural number $L,$ we have the
following equation for any $n\geq1$:
\begin{eqnarray}\label{eq:key}
&&\pr_x \bigl(\Z_{n+1}(A)=L\bigr)=\int_{\T}\pr_z \bigl(\Z_n(A)=L\bigr)h_1(x,z)\,\di z \nonumber\\ [-8pt]\\ [-8pt]
&&\hphantom{\pr_x \bigl(\Z_{n+1}(A)=L \bigr)=}{}+\int_{\T}\int_{\T} \pr_{z_1,z_2} \bigl(\Z_n(A)=L \bigr)h_2(x,z_1,z_2)\,\di z_1\,\di z_2,\nonumber
\end{eqnarray}
where $h_1(x,z)\dvtx \T\times\T\rightarrow\mathbb{R}_+$ and
$h_2(x,z_1,z_2)\dvtx \T\times\T\times\T \rightarrow\mathbb{R}_+$ are
defined as

\begin{eqnarray*}\label{eq:def_of_h_1}
h_1(x,z) =\cases{
\displaystyle f_{x,1}(z), &\quad if $\displaystyle x\in A_1^+\cap T$,\cr
\displaystyle f_{x,1}(z) + 2f_{x,2}(z)\biggl(1-\int_{\T}f_{x,4}(y)\,\di y \biggr), &\quad if $\displaystyle x\in A_2^+\cap T$,\cr
\ds 2f_{x,2}(z)\biggl(1-\int_{\T}f_{x,4}(y)\,\di y \biggr), & \quad if $x\in A_3\cap T$,\cr
\displaystyle f_{x,3}(z) + 2f_{x,2}(z)\biggl(1-\int_{\T}f_{x,4}(y)\,\di y \biggr), &\quad if $\displaystyle x\in A_2^-\cap T$,\cr
\displaystyle f_{x,3}(z), & \quad if $\displaystyle x\in A_1^-\cap T$
}
\end{eqnarray*}
and
\begin{eqnarray*}\label{eq:def_of_h_2}
h_2(x,z_1,z_2) = \cases{
2f_{x,2}(z_1)f_{x,4}(z_2), & \quad if $x\in(A_3\cup A_2^+\cup A_2^-)\cap T$,\cr
0, & \quad otherwise.
}
\end{eqnarray*}
Both are bounded and piecewise uniformly continuous functions in
$x$ on $\T$ for any fixed $z,z_1,z_2\in\T$.
\end{lemma}

\begin{pf}
The decomposition (\ref{eq:key}) is obtained from the Chapman--Kolmogorov equation, that is, by conditioning on the first generation.
In the corresponding formula (\ref{eq:Kolmogorov}), we use one of the
conclusions of Lemma \ref{lem:at_most_two}, that is, that exactly two
squares in generation 1 can only be generated by $Q_2$ and $Q_4$:
\begin{eqnarray}\label{eq:Kolmogorov}
&&\pr_x \bigl(\Z_{n+1}(A)=L\bigr)=\int_{T}\pr_{z}\bigl(\Z_n(A)=L\bigr)\pr_x \bigl(\Z_1(\di
z)=1\bigr)\nonumber\\ [-8pt]\\ [-8pt]
&&\hphantom{\pr_x \bigl(\Z_{n+1}(A)=L\bigr)=}{}+\int_{T}\int_{T}\pr_{z_1,z_2}\bigl(\Z_n(A)=L\bigr)
\pr_x\bigr(\mathcal{Z}^2_1(\di z_1)=1,\mathcal{Z}^4_1(\di
z_2)=1\bigr).\nonumber
\end{eqnarray}
We have to determine the density function $h_1(x,z)$ of
exactly one descendant with type $\di z$ and the
density function $h_2(x,z_1,z_2)$ of exactly two descendants with type
$\di z_1\,\di z_2$. To perform the computation, we note that the statement of Lemma
\ref{lem:at_most_two} remains valid if we replace $\Phi_i(x)$ by
$X_i(x)$ because of the definition of $X_i(x)$ in equation (\ref{def_X_i(x)}).
One can decompose the probability of having exactly one descendant
such that the type of this descendant falls into the set
$(-\infty,z]$ (for any real $z$) as follows:
\[
\pr_x \bigl(\Z_1 ( (-\infty,z])=1 \bigr)=\sum_{i=1}^4\pr\bigl(X_i(x)\in(-\infty,z] , X_j(x)= \nex, \forall j\ne i\bigr).
\]

The decomposition in Lemma \ref{lem:at_most_two}, together with the
remark in the first paragraph of this proof, implies that $\{X_2(x)\neq
\nex\}\cup\{ X_4(x)\neq\nex\}$, $\{X_1(x)\neq\nex\}$ and $\{
X_3(x)\neq\nex\}$ are disjoint events for any $x\in\T$. Therefore,
one obtains
\begin{eqnarray*}
&&\pr_x \bigl(\Z_1((-\infty,z])=1 \bigr)=\pr\bigl(X_1(x)\in(-\infty,z] \bigr)\\
&&\hphantom{\pr_x \bigl(\Z_1((-\infty,z])=1 \bigr)=}{}+2\pr\bigl(X_2(x)\in(-\infty,z]\bigr) \pr \bigl( X_4(x)=\nex\bigr)\\
&&\hphantom{\pr_x \bigl(\Z_1((-\infty,z])=1 \bigr)=}{}+\pr\bigl( X_3(x)\in(-\infty,z] \bigr),
\end{eqnarray*}
using the fact that $X_2(x)$ and $X_4(x)$ are independent and
identically distributed.
Since $X_i(x)$ has density $f_{x,i}$, one gets that this equals
\begin{eqnarray*}
&&\int_{(-\infty,z]}f_{x,1}(y)\,\di y\cdot\ind_{(A_1^+\cup A_2^+)\cap T}(x)\\
&&\quad{}+2\int_{(-\infty,z]}f_{x,2}(y)\,\di y\,\biggl(1-\int_{\T} f_{x,4}(y)\,\di y \biggr)\cdot\ind_{(A_3\cup A_2^+\cup A_2^-)\cap T}(x)\\
&&\quad{}+\int_{(-\infty,z]}f_{x,3}(y)\,\di y\cdot\ind_{(A_1^-\cup A_2^-)\cap T}(x)\\
&&\qquad=\int_{(-\infty,z]}h_1(x,y)\,\di y.
\end{eqnarray*}

Let us next deal with exactly two descendants with types falling into
$(-\infty,z_1]$ (resp. $(-\infty,z_2]$).
This probability equals
\[
2\pr\bigl( X_2(x)\in(-\infty,z_1],X_4(x)\in(-\infty,z_2] \bigr).
\]
Since $X_2(x)$ and $X_4(x)$ are independent and identically
distributed, one obtains that this equals
\begin{eqnarray*}
&&2\int_{(-\infty,z_1]}f_{x,2}(y)\,\di y\,\int_{(-\infty,z_2]}f_{x,4}(y)\,\di y\cdot\ind_{( A_3\cup A_2^+\cup A_2^-)\cap T}(x)\\
&&\qquad=\int_{(-\infty,z_1]}\int_{(-\infty,z_2]}h_2(x,y_1,y_2)\,\di y_1\,\di y_2.
\end{eqnarray*}
Summarizing these considerations, one obtains (\ref{eq:key}).

The piecewise continuity of $h_1(x,z)$ and $h_2(x,z_1,z_2)$ in $x$
follows from the definitions of $h_1$ and $h_2$, respectively. Since
they have compact support, $h_1$ and $h_2$ are piecewise uniformly
continuous in $x$.
\end{pf}

Let $A\subset\T$ such that the Lebesgue measure of $A$ is positive.
Let $W_n(A)=\Z_n(A)\rho^{-n}$ and $W(A)=\lim_{n\to\infty}W_n(A),$ which
almost surely exists by Theorem~\ref{thmA:14.1}.
We need a stronger result: the random variable $W(A)$ is strictly
separated from 0
with uniformly positive probability for some neighborhood of the
initial type 0.
This is shown in the next lemma.

\begin{lemma}\label{lem:W_uniformly_positive}
For some neighborhood $J\subset\T$ of 0 and positive numbers $y$ and~$r,$ we have
\begin{equation}\label{eq:lem_W_uniformly_positive}
\inf_{x\in J}\pr_x \bigl(W(A)> y\bigr)\geq r.
\end{equation}
\end{lemma}

\begin{pf}
Lemma \ref{lem:decomp_in_the_first_generation} implies that
\begin{eqnarray}\label{eq:W_n_H_n}
\pr_x \bigl(W_{n+1}(A)\leq y\bigr)&=&\pr_x\bigl(\Z_{n+1}(A)\leq\rho^{n+1}y \bigr)\nonumber\\
&=&\int_{\T}\pr_z \bigl( W_n(A)\leq\rho y\bigr)h_1(x,z)\,\di z \\
&&{}+\int_{\T}\int_{\T} \pr_{z_1,z_2} \bigl( W_n(A)\leq\rho y\bigr)h_2(x,z_1,z_2)\,\di z_1\,\di z_2.\nonumber
\end{eqnarray}
We will investigate the convergence of the last two terms in (\ref{eq:W_n_H_n}).

Theorem \ref{thmA:14.1} implies that we have, for all $z\in\T,$
\begin{equation}\label{eq:weak_convergence_of_P_z}
\lim_{n\to\infty}\pr_z \bigl(W_n(A)\leq y\bigr)=\pr_z \bigl(W(A)\leq y\bigr)
\end{equation}
if $y\in \operatorname{Cont}(\pr_{z,A}),$ where $\operatorname{Cont}(\pr_{z,A})$ denotes the set of
continuity
points of the distribution function on the right-hand side of (\ref
{eq:weak_convergence_of_P_z}).

Next, we seek the weak convergence of the measure $\pr_{z_1,z_2}(W_n(A)\in\cdot)$, which is the convolution of the measures
$\pr_{z_1}(W_n(A)\in\cdot)$ and $\pr_{z_2}
(W_n(A)\in\cdot)$. Since they are weakly convergent, the
convolution is also weakly convergent. So,
\begin{equation}\label{eq:weak_convergence_of_P_z1_z2}
\lim_{n\to\infty}\pr_{z_1,z_2} \bigl(W_n(A)\leq y\bigr)=\pr_{z_1,z_2}\bigl(W(A)\leq y\bigr)
\end{equation}
if $y\in \operatorname{Cont}(\pr_{z_1,z_2,A})$.

Let, for $z,z_1,z_2\in\T$, $y>0$ and $\eps$ a small positive number
(to be chosen later), $t_{y}:=t(z,z_1,z_2;y, \eps)$ be a real number
such that
\[
y\leq t_y< y+\eps\quad\mbox{and}\quad\rho t_y\in \operatorname{Cont}(\pr_{z,A})\cap \operatorname{Cont}(\pr_{z_1,z_2,A}),
\]
and let us define the following two functions:
\begin{eqnarray*}\label{eq:def_of_beta_n_and_beta}
\theta_{n+1} (x,y,A)&=&\int_{\T}\pr_z \bigl( W_n(A)\leq\rho t_y\bigr)h_1(x,z)\,\di z\\
&&{}+ \int_{\T}\int_{\T} \pr_{z_1,z_2} \bigl(W_n(A)\leq\rho t_y\bigr)h_2(x,z_1,z_2)\,\di z_1\,\di z_2,\\
\theta(x,y,A)&=&\int_{\T}\pr_z \bigl( W(A)\leq\rho t_y\bigr)h_1(x,z)\,\di z \\
&&{}+\int_{\T}\int_{\T} \pr_{z_1,z_2} \bigl(W(A)\leq\rho t_y\bigr)h_2(x,z_1,z_2)\,\di z_1\,\di z_2 .
\end{eqnarray*}
Using the decomposition (\ref{eq:W_n_H_n}), the definition of $t_y$
and the right-continuity
of distribution functions, we can derive the following bounds:
\[
\pr_x \bigl(W_{n+1}(A)\leq y\bigr)\leq\theta_{n+1}(x,y,A) \leq\pr_x
\bigl(W_{n+1}(A)\leq y+\eps\bigr).\label{eq:bound_to_beta_n(x,eps)}
\]
By using (\ref{eq:weak_convergence_of_P_z}), (\ref
{eq:weak_convergence_of_P_z1_z2})
and the bounded convergence theorem, we get that $\theta_n(x,y,A)$
converges as $n\to\infty,$ so
\begin{equation}\label{eq:bound_to_beta(x,eps)}
\pr_x \bigl(W(A)\leq y\bigr)\leq\theta(x,y,A) \leq\pr_x \bigl(W(A)\leq y+\eps\bigr).
\end{equation}
Using the piecewise continuity of $h_1$ and $h_2$ in $x$ (Lemma \ref{lem:decomp_in_the_first_generation})
 and bounded convergence, one can
see that $\theta_{n} (x,y,A)$ and $\theta(x,y,A)$ are piecewise
continuous on $\T$ in $x$.

Using inequality (\ref{eq:star}) in Theorem \ref{thmA:14.1} and the
right-continuity of distribution functions, we can
find two positive numbers $r,u$ such that $\pr_0(W(A)>u)>2r$ or, equivalently,
$\pr_0(W(A)\leq u)\leq1-2r$.
Let $y=u-\eps$ for some positive $\eps<u$. Using the second
inequality of (\ref{eq:bound_to_beta(x,eps)}), one gets $\theta
(0,y,A)\leq\pr_0(W(A)\leq y+\eps)\leq1-2r$.
Since $\theta(x,y,A)$ is piecewise continuous on $\T,$ there exist an
interval $J\subset\T$ which is a neighborhood of 0 such that the
bound $\theta(x,y,A)$ is uniformly smaller than 1 on this interval,
that is,
$\sup_{x\in J}\theta(x,y,A)\leq1-r$.
The first inequality of (\ref{eq:bound_to_beta(x,eps)})
implies that $ \sup_{x\in J} \pr_x (W(A)\leq y)\leq\sup_{x\in
J}\theta(x,y,A)\leq1-r,$ which yields
the required bound in (\ref{eq:lem_W_uniformly_positive}).
\end{pf}

\begin{lemma}\label{lem:p2_uniformly_positive}
There exist two positive numbers $\eta$, $r$, an integer $N$ and a
number $K$ with $0<K<\frac18$ such that
\[
\inf_{n\geq N}\inf_{x\in[-K,K]}\pr_x \bigl(\Z_n([-K,K ])>\eta\rho^n \bigr)
> \frac{r}{2}.
\]
\end{lemma}
\begin{pf}
We apply Lemma \ref{lem:W_uniformly_positive} with $A=\T$ and obtain
the numbers $y$, $r$ and the set $J$. Let $K$ be a positive number such
that $K<\frac18$ and $[-K,K]\subset J$. We then have
\[
\inf_{x\in[-K,K]}\pr_x \bigl(W(\T)> y\bigr)\geq r.
\]
Using Theorem \ref{thmA:14.1}, we get that
\[
W([-K,K])=\gamma W(\T)
\]
holds $\pr_x$ almost surely for any $x\in\T$, where
\[
\gamma=\frac{\int_{[-K,K]}\nu(z)\,\di z}{\int_{\T}\nu(z)\,\di z}.
\]
Hence, we have the bound
\[
\inf_{x\in[-K,K]}\pr_x \bigr(W([-K,K])> \eta+\eps\bigl)> r,
\]
where $\eta+\eps=\gamma y$ for some positive $\eta$ and $\eps$.
This and the second inequality of (\ref{eq:bound_to_beta(x,eps)})
together imply that $\theta(x,\eta,[-K,K ])$ is uniformly smaller
than 1:
\begin{equation}\label{eq:delta+eps}
\sup_{x\in[-K,K ]}\theta(x,\eta,[-K,K ])\leq\sup_{x\in[-K,K]}\pr_x \bigl(W([-K,K])\leq\eta+\eps\bigr)\leq1- r.
\end{equation}

We will show that $\theta_n(x,\eta,[-K,K ])$ converges uniformly to
$\theta(x,\eta,[-K,K ])$ on $[-K,K]$ as $n$ tends to infinity.
Writing
\[
E_n:=W_n([-K,K ])\leq\rho\eta_t\quad\mbox{and}\quad E:=W([-K,K])\leq\rho\eta_t,
\]
using trivial estimations, one gets the following chain of inequalities:
\begin{eqnarray*}
&&\sup_{x\in[-K,K ]}|\theta_{n+1}(x,\eta,[-K,K ])-\theta(x,\eta,[-K,K ] )|\\
&&\qquad\leq\sup_{x\in[-K,K ]}\int_{\T}|\pr_z (E_n) \pr_z ( E)|h_1(x,z)\,\di z \\
&&\qquad\quad{}+\sup_{x\in[-K,K ]}\int_{\T}\int_{\T}| \pr_{z_1,z_2} (E_n) - \pr_{z_1,z_2}(E) | h_2(x,z_1,z_2)\,\di z_1\,\di z_2\\
&&\qquad\leq\sup_{x,z\in\T}h_1(x,z)\cdot\int_{\T}|\pr_z(E_n)- \pr_z (E)|\,\di z\\
&&\qquad\quad{}+\sup_{x,z_1,z_2\in\T}h_2(x,z_1,z_2)\cdot\int_{\T}\int_{\T}|\pr_{z_1,z_2} (E_n) -\pr_{z_1,z_2}(E)|\,\di z_1\,\di z_2.
\end{eqnarray*}

By bounded convergence, both integrals in the last expression converge
to~0. The suprema are finite since $h_1$ and $h_2$ are bounded (see
Lemma \ref{lem:W_uniformly_positive}). So, $\theta_n (x,\eta,[-K,K
])$ uniformly converges to $\theta(x,\eta,[-K,K ])$ on $[-K,K ]$.
Therefore, there exists an index $N$ such that for $n\geq N,$
\[
\sup_{x\in[-K,K ]}|\theta_{n}(x,\eta,[-K,K ])-\theta(x,\eta,[-K,K] )|\leq\frac{r}{2}.
\]
Using the first inequality of (\ref{eq:bound_to_beta(x,eps)}), the
triangular inequality, (\ref{eq:delta+eps}) and Lemma \ref
{lem:W_uniformly_positive}, one can write
\begin{eqnarray*}
\sup_{x\in[-K,K ]}\pr_x \bigl(W_n([-K,K])\leq\eta\bigr)&\leq&\sup_{x\in[-K,K ]}\theta_n(x,\eta,[-K,K ])\\
&\leq&\sup_{x\in[-K,K ]}\theta(x,\eta,[-K,K ])\\
&&\hspace{-3pt}{}+\sup_{x\in[-K,K ]}|\theta_{n}(x,\eta,[-K,K ])\\
&&\hspace{51pt}{}-\theta(x,\eta,[-K,K] )|\\
&\leq&1-r+\frac{r}{2}=1-\frac{r}{2}
\end{eqnarray*}
for $n\geq N$. This gives the conclusion of the lemma.
\end{pf}

\section{\texorpdfstring{The proof of the \protect\hyperref[15]{Main Lemma}}{The proof of the Main Lemma}}\label{sec:main}
We first repeat the \hyperref[15]{Main Lemma}.

\begin{main*}
There exist three positive numbers $\delta$, $q$, $K$ and an index $N$
such that
\[
\inf_{n>N}\inf_{x\in[-K,K]}\pr_x \bigl(\Z_n([0,K])>\delta\rho^n\,\&\,\Z_n([-K,0])>\delta\rho^n \bigr)> q.
\]
\end{main*}

\begin{pf}
Take $K$ as defined in Lemma \ref{lem:p2_uniformly_positive}.
Since $[-K,K]=[-K,0]\cup[0,K]$ and type 0 has probability 0 to occur,
it follows
directly from Lemma \ref{lem:p2_uniformly_positive} that one of $\pr
_x (\Z_n([0,K])>\delta\rho^n)$ and
$\pr_x (Z_n([-K,0])> \delta\rho^n )$ is larger than
$r/4$ for all $x\in[-K,K]$ and $n>N$.
But, then, by symmetry, \textit{both} of these probabilities are larger
than $r/4$.

Now, take any $x\in[-K,K]$. Since $K<\frac18$, it follows that with a
positive probability denoted by $p_{2,4}$, in the first generation, the
squares $Q_2$ and $Q_4$---with respective types $x_2$ and $x_4$ from a
subinterval of $[-K,K]$---will be present. But, by the above, these two
squares will, independently of each other and with probability at least
$r/4$, generate more than $\delta\rho^n$ squares with type in $[0,K]$
(resp. $[-K,0]$) in generation $n+1$. Thus, for all $x\in[-K,K]$ and $n>N,$
\[
\pr_x \bigl(\Z_{n+1}([0,K])>\delta\rho^n\,\&\,\Z_{n+1}([-K,0])>\delta\rho^n \bigr)> p_{2,4}\cdot\frac{r}4\cdot\frac{r}4.
\]
So, replacing $\delta$ by $\delta/\rho$, $N$ by $N+1$ and defining
$q=p_{2,4}r^2/16$, this proves the \hyperref[15]{Main Lemma}.
\end{pf}

\section*{Acknowledgment}

We wish to thank an anonymous referee for meticulously reading our
paper and proposing a large number of valuable improvements.


\printaddresses

\end{document}